\documentclass[12pt,a4paper,leqno]{amsart}

\usepackage[latin1]{inputenc}
\usepackage[T1]{fontenc}
\usepackage{amsfonts}
\usepackage{amsmath}
\usepackage{amssymb}
\usepackage{eurosym}
\usepackage{mathrsfs}
\usepackage{palatino}
\usepackage{color}
\usepackage{esint}
\usepackage{url}

\newcommand{\R}{\mathbb{R}}
\newcommand{\C}{\mathbb{C}}

\newcommand{\N}{\mathbb{N}}

\newcommand{\Z}{\mathbb{Z}}

\numberwithin{equation}{section}

\newcommand{\ud}[0]{\,\mathrm{d}}

\newcommand{\abs}[1]{|#1|}

\newcommand{\Norm}[2]{\|#1\|_{#2}}

\newcommand{\pair}[2]{\langle #1,#2 \rangle}

\newcommand{\ave}[1]{\langle #1\rangle}

\newcommand{\bddlin}[0]{\mathcal{L}}

\newcommand{\BMO}[0]{\operatorname{BMO}}
\newcommand{\supp}[0]{\operatorname{spt}}

\swapnumbers
\theoremstyle{plain}
\newtheorem{thm}[equation]{Theorem}
\newtheorem{lem}[equation]{Lemma}
\newtheorem{prop}[equation]{Proposition}

\theoremstyle{definition}
\newtheorem{defn}[equation]{Definition}

\newtheorem{ass}[equation]{Assumption}

\theoremstyle{remark}
\newtheorem{rem}[equation]{Remark}

\title{Non-homogeneous $T1$ theorem for bi-parameter singular integrals}

\author{Tuomas Hyt\"onen}
\address{Department of Mathematics and Statistics, University of Helsinki, P.O.B. 68, FI-00014 Helsinki, Finland}
\email{tuomas.hytonen@helsinki.fi}
\thanks{T.H. is supported by the European Union through the ERC Starting Grant Analytic-probabilistic methods for borderline singular integrals, and by the Academy of Finland, grants 130166 and 133264.
H.M. is supported by the Emil Aaltonen Foundation, and by the Academy of Finland, grants 130166 and 133264. This work was started when H.M. was still at the University of Helsinki.}

\author{Henri Martikainen}
\address{D\'epartement de Math\'ematiques, B\^atiment 425, Facult\'e des Sciences d'Orsay, Universit\'e Paris-Sud 11, F-91405 Orsay Cedex}
\email{henri.martikainen@math.u-psud.fr}

\makeatletter
\@namedef{subjclassname@2010}{%
  \textup{2010} Mathematics Subject Classification}
\makeatother

\subjclass[2010]{42B20}
\keywords{Bi-parameter singular integral, non-homogeneous analysis}

\begin{document}
\maketitle

\begin{abstract}
We prove a non-homogeneous $T1$ theorem for certain bi-parameter singular integral operators. Moreover, we discuss the related non-homogeneous Journ\'e's lemma and product BMO theory.
\end{abstract}

\tableofcontents

\section{Introduction}

This paper deals with singular integrals with two attributes: `bi-parameter' and `non-homogeneous'. Both of them have been investigated before but, as far as we know, only one at a time. Here, for the first time, we work in the simultaneous presence of both complications.

To be more precise, we study bi-parameter singular integrals $T$ acting on some class of functions with product domain $\R^{n+m} = \R^n \times \R^m$.
However, instead of the Lebesgue measure we equip $\R^{n+m}$ with a product measure $\mu = \mu_n \times \mu_m$, where the measures $\mu_n$ and $\mu_m$ are only assumed to be upper doubling (a condition that is more general than the assumption $\mu_n(B(x,r)) \le Cr^s$, $\mu_m(B(x,r)) \le Cr^t$). We establish a $T1$ theorem, i.e. a boundedness criterion for singular integral operators, in this setting.

Such a result touches on quite a few topics. We now lay down some of the context. After the classical $T1$ and $Tb$ type theory by David and Journ\'e \cite{DJ} and
David, Journ\'e and Semmes \cite{DJS}, the first $T1$ type theorem for product spaces was proved by Journ\'e \cite{Jo2}. Journ\'e formulated the assumptions in the language of vector-valued Calder\'on--Zygmund theory. Recently, Journ\'e's result was challenged by a new product-space $T1$ theorem by S. Pott and P. Villarroya \cite{PV}, who introduced an alternative framework avoiding the vector-valued assumptions.  Following a similar philosophy, the latter author proved a bi-parameter analog \cite{Ma1} of the representation theorem of the first author \cite{Hy2} -- a particular consequence of which is yet another bi-parameter $T1$ theorem of a more dyadic flavour than \cite{PV}.

All the classical multiparameter methods including Journ\'e's covering lemma \cite{Jo1} and the product BMO, Hardy space and multiparameter singular integral theory by 
Chang and Fefferman \cite{CF1}, \cite{CF2}, Fefferman \cite{Fe} and Fefferman and Stein \cite{FS} are in the doubling situation. The same is true for the results
of the previous paragraph, and for the other related modern developments like \cite{CLMP}, \cite{LM}, \cite{PW} and \cite{Tr}. These deal with
variations of Journ\'e's covering lemma, multiparameter paraproducts and dyadic versions of multiparameter function spaces. Since all the relevant results and definitions, both classical and modern, are in the doubling situation,
we have to consider Journ\'e's lemma, (dyadic) product BMO, $H^1$-BMO type duality results, singular integrals and multiparameter paraproducts all from the new perspective of general measures.

The following explains the specific need for these types of results.
Even in the one-parameter theory of non-homogeneous singular integrals, one needs to replace the familiar BMO condition by
\begin{displaymath}
\int_I |f-f_I|^p\,d\mu \le L\mu(\kappa I),
\end{displaymath}
where a parameter $\kappa>1$ specifies an expansion $\kappa I$ of the cube $I$; this can have a much larger measure than $I$ when $\mu$ is non-doubling.
Accordingly, we need to define and study a new space BMO$_{\textup{prod}}(\mu)$, which is used in the formulation of our main theorem.
To show that the new condition $T1 \in \BMO_{\textup{prod}}(\mu)$ is necessary, we need a version of Journ\'e's covering lemma for general product measures.
Finally, to handle some mixed paraproducts, we need a certain new $H^1$--BMO type duality inequality.

Our basic proof strategy of the $T1$ theorem uses dyadic probabilistic techniques adapted to the bi-parameter situation. These were already used by the latter author in \cite{Ma1} but, again, only in the doubling case.
The dyadic and probabilistic methods of non-homogeneous analysis were pioneered by Nazarov, Treil and Volberg (see e.g. the non-homogeneous $Tb$ theorems \cite{NTVa} and \cite{NTV}).
However, these powerful tools are not widely used in multiparameter harmonic analysis, and some extra care is needed. In this regard the current paper is a continuation of the recent developments (see e.g.
\cite{Hy1}, \cite{Hy2}, \cite{Hy3}, \cite{HM}, \cite{Ma1}, \cite{Ma2}) in the probabilistic methods. 

The proof starts by expanding in a product Haar basis adapted to the general measures $\mu_n$ and $\mu_m$. The summation in the bi-parameter case is rather massive
and the non-homogeneous techniques are needed in various parts of the summation. All kinds of bad-boundary terms appear -- also in some new mixed situations. Moreover,
there is a wide variety of new non-homogeneous paraproducts, some of them related to the new space BMO$_{\textup{prod}}(\mu)$.

Just like in \cite{Ma1} our operators are defined using the philosophy of Pott and Villarroya
\cite{PV}. That is, we mostly avoid the language of vector-valued formulations used in the original work of Journ\'e \cite{Jo2}. Of course, our kernel estimates are tied to the non-homogeneous measures. 
\section{The main theorem} 

We use this section to introduce the (somewhat lengthy) framework necessary for the formulation of our main theorem. We consider the following class of measures:

\begin{defn}[Upper doubling measures]
Let $\lambda\colon \R^n \times (0,\infty) \to (0,\infty)$ be a function so that $r \mapsto \lambda(x,r)$ is non-decreasing and $\lambda(x, 2r) \le C_{\lambda}\lambda(x,r)$ for
all $x \in \R^n$ and $r > 0$. We say that a Borel measure $\mu$ in $\R^n$ is upper doubling with the dominating function $\lambda$, if $\mu(B(x,r)) \le
\lambda(x,r)$ for all $x \in \R^n$ and $r > 0$. We set $d_{\lambda} = \log_2 C_{\lambda}$.
\end{defn}

The property $\lambda(x, |x-y|) \sim \lambda(y, |x-y|)$ would be convenient. This is luckily something that can be arranged for free.
In \cite[Proposition 1.1]{HYY} it is shown that
$\Lambda(x,r) := \inf_{z \in \R^n} \lambda(z, r + |x-z|)$ satisfies that $r \mapsto \Lambda(x,r)$ is non-decreasing, $\Lambda(x,2r) \le C_{\lambda}\Lambda(x,r)$,
$\mu(B(x,r)) \le \Lambda(x,r)$, $\Lambda(x,r) \le \lambda(x,r)$ and $\Lambda(x,r) \le C_{\lambda}\Lambda(y,r)$ if $|x-y| \le r$. Therefore, we may (and do) always assume that dominating functions $\lambda$ satisfy the additional symmetry property
$\lambda(x,r) \le C\lambda(y,r)$ if $|x-y| \le r$.


Henceforth, let $\mu = \mu_n \times \mu_m$, where $\mu_n$ and $\mu_m$ are upper doubling measures on $\R^n$ and $\R^m$ respectively. The corresponding dominating functions
are denoted by $\lambda_n$ and $\lambda_m$. We use $\ell^{\infty}$ metrics on $\R^n$ and $\R^m$.

\subsection{Bi-parameter Calder\'on--Zygmund  operators}

We study an a priori bounded linear operator $T\colon L^2(\mu) \to L^2(\mu)$.
In addition to the usual adjoint $T^*$, we will be concerned with the \emph{partial adjoint}  $T_1$ defined by
\begin{displaymath}
  \langle T_1(f_1 \otimes f_2), g_1 \otimes g_2 \rangle = \langle T(g_1 \otimes f_2), f_1 \otimes g_2\rangle.
\end{displaymath}
The philosophy of the following assumptions is that we impose symmetric conditions on all the four operators $T$, $T^*$, $T_1$ and $T_1^*$.

\begin{ass}[Full kernel representation]
If $f = f_1 \otimes f_2$ and $g= g_1 \otimes g_2$ with $f_1, g_1 \colon \R^n \to \C$, $f_2, g_2 \colon \R^m \to \C$, $\textrm{spt}\,f_1\cap \textrm{spt}\,g_1 = \emptyset$ and $ \textrm{spt}\,f_2 \cap  \textrm{spt}\,g_2 = \emptyset$,
we have the kernel representation
\begin{equation*}
\begin{split}
  \langle Tf, g\rangle &= \int_{\R^{n+m}}\int_{\R^{n+m}} K(x,y)f(y)g(x)\,d\mu(x)\,d\mu(y) \\
   &=\iint_{\R^n\times\R^m}\iint_{\R^n\times\R^m} K(x_1,x_2;y_1,y_2)f(y_1,y_2)g(x_1,x_2)\,d\mu(x_1,x_2)\,d\mu(y_1,y_2),
\end{split}
\end{equation*}
\begin{displaymath}
\end{displaymath}
where the kernel $K$ is a function
\begin{equation*}
  K\colon (\R^{n+m} \times \R^{n+m}) \setminus \{(x,y) \in \R^{n+m}  \times \R^{n+m}:\, x_1 = y_1 \textrm{ or } x_2 = y_2\} \to \C.
\end{equation*}
\end{ass}

Note that this implies kernel representations for $T^*$, $T_1$ and $T_1^*$. If we denote their kernels respectively by $K^*$, $K_1$ and $K_1^*$, then we immediately observe the formulae
\begin{equation*}
\begin{split}
  K^*(x,y) &= K(y_1,y_2;x_1,x_2), \\
  K_1(x,y) &= K(y_1,x_2;x_1,y_2), \\
  K_1^*(x,y) &= K(x_1,y_2;y_1,x_2).
\end{split}
\end{equation*}

\begin{ass}[Full standard estimates]
The kernel is assumed to satisfy the \emph{size condition}
\begin{displaymath}
|K(x,y)| \le C\frac{1}{\lambda_n(x_1, |x_1-y_1|)}\frac{1}{\lambda_m(x_2, |x_2-y_2|)},
\end{displaymath}
the \emph{H\"older condition}
\begin{align*}
|K(x,y) - K(x, &(y_1, y_2')) - K(x, (y_1', y_2)) + K(x, y')| \\
&\le  C \frac{|y_1-y_1'|^{\alpha}}{|x_1-y_1|^{\alpha}\lambda_n(x_1, |x_1-y_1|)}  \frac{|y_2-y_2'|^{\beta}}{|x_2-y_2|^{\beta}\lambda_m(x_2, |x_2-y_2|)} 
\end{align*}
whenever $|y_1-y_1'| \le |x_1-y_1|/2$ and $|y_2-y_2'| \le |x_2-y_2|/2$,
%
and the \emph{mixed H\"older and size condition}
\begin{align*}
|K(x,y) - K(x, (y_1', y_2))| \le C \frac{|y_1-y_1'|^{\alpha}}{|x_1-y_1|^{\alpha}\lambda_n(x_1, |x_1-y_1|)} \frac{1}{\lambda_m(x_2, |x_2-y_2|)}
\end{align*}
whenever $|y_1-y_1'| \le |x_1-y_1|/2$.
The same conditions are imposed on $K^*$, $K_1$ and $K_1^*$ as well.
\end{ass}

We need to assume some Calder\'on--Zygmund structure on $\R^n$ and $\R^m$ separately.

\begin{ass}[Partial kernel representation]
If $f = f_1 \otimes f_2$ and $g = g_1 \otimes g_2$ with $\textrm{spt}\,f_1\cap \textrm{spt}\,g_1 = \emptyset$, then we assume the kernel representation
\begin{displaymath}
\langle Tf, g\rangle = \int_{\R^n} \int_{\R^n} K_{f_2, g_2}(x_1,y_1)f_1(y_1)g_1(x_1)\,d\mu_n(x_1)\,d\mu_n(y_1).
\end{displaymath}
\end{ass}

\begin{ass}[Partial boundedness $\times$ standard estimates]
The kernel
\begin{equation*}
  K_{f_2, g_2} \colon (\R^n \times \R^n) \setminus \{(x_1, y_1) \in \R^n \times \R^n:\, x_1 = y_1\}
\end{equation*}
is assumed to satisfy the size condition
\begin{displaymath}
|K_{f_2,g_2}(x_1,y_1)| \le C(f_2,g_2)\frac{1}{\lambda_n(x_1, |x_1-y_1|)}
\end{displaymath}
and the H\"older conditions
\begin{displaymath}
|K_{f_2,g_2}(x_1,y_1) - K_{f_2,g_2}(x_1',y_1)| \le C(f_2,g_2)\frac{|x_1-x_1'|^{\alpha}}{|x_1-y_1|^{\alpha}\lambda_n(x_1, |x_1-y_1|)}
\end{displaymath}
whenever $|x_1 - x_1'| \le |x_1-y_1|/2$, and
\begin{displaymath}
|K_{f_2,g_2}(x_1,y_1) - K_{f_2,g_2}(x_1,y_1')| \le C(f_2,g_2)\frac{|y_1-y_1'|^{\alpha}}{|x_1-y_1|^{\alpha}\lambda_n(x_1, |x_1-y_1|)}
\end{displaymath}
whenever $|y_1 - y_1'| \le |x_1-y_1|/2$.

We assume that $C(f_2, g_2) \lesssim \|f_2\|_{L^2(\mu_m)} \|g_2\|_{L^2(\mu_m)}$.
\end{ass}

We assume the analogous representation and properties with a kernel
$K_{f_1, g_1}$ in the case spt$\,f_2 \cap \textrm{spt}\, g_2 = \emptyset$.

\begin{rem}
As is clear, the final bound is not allowed to depend on $\|T\|_{L^2(\mu) \to L^2(\mu)}$.
However, note that we do assume some type of $L^2$ boundedness separately on $(\R^n, \mu_n)$ and $(\R^m, \mu_m)$ (the bounds for $C(f_1, g_1)$ and $C(f_2, g_2)$).
This is in slight contrast with some of the most recent works in the homogeneous bi-parameter setting \cite{Ma1}. There one may work with a bit weaker testing conditions for $C(\cdot, \cdot)$.
However, this seems to be a standard assumption in the classical homogeneous works, and we also require it in our non-homogeneous setting.
\end{rem}

\begin{ass}[Full weak boundedness property]
We assume the following weak boundedness property: for every cube $I \subset \R^n$ and $J \subset \R^m$ there holds that
\begin{displaymath}
|\langle T(\chi_I \otimes \chi_J), \chi_I \otimes \chi_J\rangle| \le C\mu_n(5I)\mu_m(5J).
\end{displaymath}
\end{ass}

Note that this condition is invariant with respect to the replacement of $T$ by either $T^*$, $T_1$ or $T_1^*$.

Moreover, we need some type of diagonal BMO assumptions. This amounts to replacing exactly one of the indicator functions above by a cancellative function supported on the same cube. More precisely:

\begin{ass}[Partial weak boundedness $\times$ BMO condition]
Let $I \subset \R^n$ and $J \subset \R^m$ be cubes and $a_I$, $a_J$ be functions such that $\int a_I \,d\mu_n = \int a_J \,d\mu_m = 0$, spt$\, a_I \subset I$ and spt$\, a_J \subset J$. We assume that
\begin{displaymath}
|\langle T(\chi_I \otimes \chi_J), a_I \otimes \chi_J \rangle| +  |\langle T^*(\chi_I \otimes \chi_J), a_I \otimes \chi_J \rangle| \le C\|a_I\|_{L^2(\mu_n)} \mu_n(5I)^{1/2} \mu_m(5J)
\end{displaymath}
and
\begin{displaymath}
|\langle T(\chi_I \otimes \chi_J), \chi_I \otimes a_J \rangle| + |\langle T^*(\chi_I \otimes \chi_J), \chi_I \otimes a_J \rangle| \le C\mu_n(5I) \|a_J\|_{L^2(\mu_m)} \mu_m(5J)^{1/2}.
\end{displaymath}
\end{ass}
Note that no new conditions arise by replacing $T$ by $T_1$ above.
The reason for calling this a BMO condition will become more clear in the course of the proof. Observe that both the full weak boundedness and the partial weak boundedness $\times$ BMO condition are obviously necessary conditions for the boundedness of $T$ on $L^2(\mu)$. In fact, they would both follow from the following stronger diagonal testing condition
\begin{displaymath}
\|\chi_I \otimes \chi_J T(\chi_I \otimes \chi_J)\|_{L^2(\mu)} + \|\chi_I \otimes \chi_J T^*(\chi_I \otimes \chi_J)\|_{L^2(\mu)} \le C\mu_n(5I)^{1/2}\mu_m(5J)^{1/2}.
\end{displaymath}

\subsection{Dyadic grids and Haar functions}

In order to formulate the actual ``$T1\in\textup{BMO}$'' conditions in the present setting, it is is necessary to introduce some notation related to dyadic martingale difference decompositions and Haar functions.

Let $\mathcal{D}^0_n$ and $\mathcal{D}^0_m$ be the standard dyadic grids on $\R^n$ and $\R^m$ respectively.

\subsubsection{Random dyadic grids and good/bad cubes}
We let $w_n = (w_{n,i})_{i \in \Z}$, $w_n' = (w'_{n,i})_{i \in \Z}$, $w_m =  (w_{m,j})_{j \in \Z}$
and $w'_m =  (w'_{m,j})_{j \in \Z}$, where $w_{n,i}, w'_{n,i} \in \{0,1\}^n$ and $w_{m,j}, w'_{m,j} \in \{0,1\}^m$. 
In $\R^n$ we define the new dyadic grid $\mathcal{D}_n = \{I + \sum_{i:\, 2^{-i} < \ell(I)} 2^{-i}w_{n,i}: \, I \in \mathcal{D}_n^0\} = \{I + w_n: \, I \in \mathcal{D}_n^0\}$, where we simply have defined
$I + w_n := I + \sum_{i:\, 2^{-i} < \ell(I)} 2^{-i}w_{n,i}$. The grids $\mathcal{D}_n'$, $\mathcal{D}_m$ and $\mathcal{D}_m'$ are similarly defined.
There is a natural product probability structure on $(\{0,1\}^n)^{\Z}$ and $(\{0,1\}^m)^{\Z}$. Even if $n=m$ and despite the notation we agree that we have four independent random dyadic grids.

A cube $I_1 \in \mathcal{D}_n$ is called bad if there exists $I_2 \in \mathcal{D}_n'$ so that $\ell(I_2) \ge 2^r \ell(I_1)$ and $d(I_1, \partial I_2) \le 4\ell(I_1)^{\gamma_n}\ell(I_2)^{1-\gamma_n}$.
Here $\ell(I)$ denotes the side length of $I$ and $\gamma_n = \alpha/(2d_{\lambda_n} + 2\alpha)$, where $\alpha > 0$ appears in the kernel estimates and, we recall, $d_{\lambda_n} =  \log_2 C_{\lambda_n}$.
The parameter $r > 0$ will be fixed later. Note also that here bad really means $\mathcal{D}_n'$-bad, but this is not usually spelled out. 
The badness in $\R^m$ is defined similarly (it involves the same parameter $r$ and $\gamma_m = \beta/(2d_{\lambda_m}+2\beta)$).

It is easy to check that, almost surely with respect to the canonical probability on $(\{0,1\}^n)^{\Z}$, a random dyadic system $\mathcal{D}_n$ has the following ``no quadrants'' property: Whenever $I_k\in\mathcal{D}_n$, $k\in\N$, is a strictly increasing sequence (meaning that $I_k\subsetneq I_{k+1}$) of dyadic cubes, then this sequence exhausts all of $\R^n$ (meaning that $\bigcup_{k\in\N}I_k=\R^n$). By throwing away a subset of the probability space with probability zero, we can and will assume that all our dyadic systems possess this additional ``no quadrants'' property. This convention is a matter of convenience, which has no effect on any of the probabilistic statements that we are going to make.

\subsubsection{Haar functions on $\R^n$ with general measures}

The content of this subsection is from \cite{Hy3}.
Let $\mathcal{D}_n$ be one of the dyadic grids on $\R^n$ as above and $\mu_n$ be any locally finite Borel measure on $\R^n$. We set
\begin{align*}
E_If  &= \langle f \rangle_I\chi_I, \\
\Delta_I f &= \sum_{I' \in \textup{ch}(I)} [ \langle f \rangle_{I'} - \langle f \rangle_I]\chi_{I'},
\end{align*}
where ch$(I) = \{I_j:\, j = 1, \ldots, 2^n\}$ is the collection of dyadic children of $I$, and $\langle f \rangle_I = \mu_n(I)^{-1} \int_I f\,d\mu_n$. For any $\ell \in \Z$ and $f \in L^2(\mu_n)$ we have the orthogonal decomposition
\begin{displaymath}
f = \sum_{\ell(I) = 2^\ell} E_I f + \sum_{\ell(I) \le 2^\ell} \Delta_I f.
\end{displaymath}
We index the children $I_j$ of $I$ in such a way that
\begin{equation*}
  \mu_n(\widehat I_k) \ge [1-(k-1)2^{-n}]\mu_n(I),\qquad
  \widehat I_{k} = \bigcup_{u=k}^{2^n} I_u.
\end{equation*}
Hence $\mu_n(\widehat I_k) \sim \mu_n(I)$ for every $k$. We then continue to decompose
\begin{align*}
E_If &= \langle f, h_{I,0}\rangle  h_{I,0}, \\
\Delta_I f& = \sum_{\eta = 1}^{2^n-1} \langle f, h_{I,\eta}\rangle  h_{I,\eta},
\end{align*}
where $h_{I,0} = \mu_n(I)^{-1/2}\chi_I$ and
\begin{displaymath}
h_{I,\eta} = \Big( \frac{\mu_n(I_{\eta})\mu_n(\widehat I_{\eta+1})}{\mu_n(\widehat I_{\eta})}\Big)^{1/2} \Big[ \frac{\chi_{I_{\eta}}}{\mu_n(I_{\eta})} - \frac{\chi_{\widehat I_{\eta+1}}}{\mu_n(\widehat I_{\eta+1})}\Big], \qquad \eta = 1, \ldots, 2^n-1.
\end{displaymath}
If $\mu_n(I_{\eta}) = 0$, we set $h_{I,\eta} = 0$.

We have the orthonormality $\langle h_{I_1, \eta_1}, h_{I_2, \eta_2} \rangle = \delta_{I_1}^{I_2} \delta_{\eta_1}^{\eta_2}$.
Moreover, for $\eta \ne 0$ we have cancellation: $\int h_{I,\eta}\,d\mu_n = 0$. Finally, $h_{I,\eta}$ is supported on $I$ and constant
on its children. We have established that $\{h_{I,0}: \, \ell(I) = 2^\ell\} \cup \{h_{I, \eta}:\,  \ell(I) \le 2^\ell, \eta \in \{1, \ldots, 2^n-1\}\}$ is
an orthonormal basis for $L^2(\mu_n)$ with the above mentioned properties.

Suppose now that $\mu = \mu_n \times \mu_m$, where $\mu_n$ and $\mu_m$ are locally finite Borel measures on $\R^n$ and $\R^m$ respectively. Assume also that we are given
a dyadic grid $\mathcal{D}_n$ on $\R^n$ and a dyadic grid $\mathcal{D}_m$ on $\R^m$. We denote the Haar functions on $\R^n$ with respect to the measure $\mu_n$
by $h_{I, \eta}$, $I \in \mathcal{D}_n$, $\eta = 0, \ldots, 2^n -1$, and we denote the Haar functions on $\R^m$ with respect to the measure $\mu_m$ by $u_{J, \kappa}$, $J \in \mathcal{D}_m$, $\kappa
= 0, \ldots, 2^m-1$.

\subsection{BMO conditions}

We recall the definition of the BMO space relevant in the upper doubling setting from \cite{NTV,HM}.

\begin{defn}
We say that $f \in L^1_{\textrm{loc}}(\mu_n)$ belongs to $\textup{BMO}^p_{\kappa}(\mu_n)$, if for any cube $I \subset \R^n$ there
exists a constant $f_I$ such that
\begin{displaymath}
\Big( \int_I |f-f_I|^p\,d\mu_n\Big)^{1/p} \le L\mu_n(\kappa I)^{1/p},
\end{displaymath}
where the constant $L$ does not depend on $I$. The best constant $L$ is denoted $\|f\|_{\textup{BMO}^p_{\kappa}(\mu_n)}$.
\end{defn}

The following lemma from \cite{NTV,HM} motivates our definition of the BMO space for a product measure $\mu_n\times\mu_m$:

\begin{lem}\label{lem:BMO1}
For all $J\in\mathcal{D}_n'$,
\begin{equation*}
  \Big(\sum_{\substack{I\in\mathcal{D}_n\textup{ good}\\ I\subset J\\ \ell(I)\leq 2^{-r}\ell(J)}}
  \abs{\ave{f,h_I}}^2\Big)^{1/2}\lesssim\mu_n(J)^{1/2}\|f\|_{\textup{BMO}_{\kappa}^2(\mu_n)}.
\end{equation*}
\end{lem}

\begin{defn}\label{def:BMOprod}
We say that $b \in \textup{BMO}_{\textup{prod}}(\mu)$, if for every $\mathcal{D} = \mathcal{D}_n \times \mathcal{D}_m$ and $\mathcal{D}' = \mathcal{D}_n' \times \mathcal{D}_m'$
there holds that
\begin{equation*}
   \Big(\sum_{\eta=1}^{2^n-1} \sum_{\kappa = 1}^{2^m-1} \sum_{\substack{S\in\mathcal{D}'\\S\subset\Omega}}
  \sum_{\substack{R = I \times J \in\mathcal{D}\textup{ good},\ R\subset S\\ \operatorname{gen}(R)=\operatorname{gen}(S)+(r,r)}}
    |\langle b, h_{I,\eta} \otimes u_{J,\kappa}\rangle|^2 \Big)^{1/2}
  \le L\mu(\Omega)^{1/2}
\end{equation*}
for all sets $\Omega \subset \R^{n+m}$ such that $\mu(\Omega) < \infty$ and such that for every $x \in \Omega$ there
exists $S \in \mathcal{D}'$ so that $x \in S \subset \Omega$. 
The best constant $L$ is denoted $\|b\|_{\textup{BMO}_{\textup{prod}}(\mu)}$. Here
\begin{displaymath}
\operatorname{gen}(R):=(\operatorname{gen}(I),\operatorname{gen}(J)), 
\end{displaymath}
and $\operatorname{gen}(I)$ is the usual generation of a dyadic cube (if $\ell(I) = 2^{-k}$ then $\operatorname{gen}(I) = k$).
Moreover, the fact that $R$ is good means that $I$ is $\mathcal{D}_n'$-good and $J$ is $\mathcal{D}_m'$-good.
\end{defn}

\begin{rem}\label{rem:BMOprod}
In the definition of $ \textup{BMO}_{\textup{prod}}(\mu)$, we can equivalently restrict to considering sets $\Omega$ which, in addition to the stated property, are also bounded. Namely, if the defining inequality holds for this more restricted class of sets, and $\Omega$ is any set as in Definition~\ref{def:BMOprod}, then we simply check the restricted product BMO condition for bounded sets of the form $\Omega\cap S_k$, where $S_k$ is the union of the $2^{n+m}$ dyadic cubes in $\mathscr{D}'$ of sidelength $2^k$ that lie closest to the origin. We conclude that the square sum over $S\subset\Omega\cap S_k$ is dominated by $L\mu(\Omega\cap S_k)^{1/2}\leq L\mu(\Omega)^{1/2}$, and taking the limit as $k\to\infty$, we find that the full sum over $S\subset\Omega$ is also dominated by $L\mu(\Omega)^{1/2}$, as required in Definition~\ref{def:BMOprod}.
\end{rem}



We can now formulate our main theorem:

\begin{thm}\label{thm:main}
Let $\mu = \mu_n \times \mu_m$, where $\mu_n$ and $\mu_m$ are upper doubling measures on $\R^n$ and $\R^m$ respectively. Let $T\colon L^2(\mu) \to L^2(\mu)$ be a bi-parameter Calder\'on--Zygmund  operator,
which is a priori bounded. Assume that $T$ satisfies all the Assumptions formulated in this section. Moreover, assume that
\begin{equation*}
  S1 \in \textup{BMO}_{\textup{prod}}(\mu)\qquad\forall S \in \{T, T^*, T_1, T_1^*\}.
\end{equation*}
Then there holds that $\|T\| \lesssim 1$, a bound depending only on the Assumptions and the $\textup{BMO}_{\textup{prod}}(\mu)$ norms of the four $S1$, 
but independent of the a priori bound.
\end{thm}

\begin{rem}
Note that at least for a large subclass of operators (see Section \ref{sec:bmo}) our conditions, except for the conditions
$T_1(1), T_1^*(1) \in \textup{BMO}_{\textup{prod}}(\mu)$, are necessary. Indeed, $T_1$ does not have to be bounded even if $T$ is. This deficit is shared
by all known bi-parameter $T1$ theorems, already in the homogeneous setting.
\end{rem}

\section{Strategy of the proof}

\subsection{Initial reductions}
The following reduction is quite standard, but we do it carefully, since in the bi-parameter case the details are not written anywhere else. We do not need
an exact averaging equality as in some of the latest non-homogeneous work -- the by now classical trick of Nazarov, Treil and Volberg suffices here.

Let us fix $f, g \in C(\R^{n+m})$ with compact support so that $0.7\|T\| \le |\langle Tf, g\rangle|$ (here $\|T\| = \|T\|_{L^2(\mu) \to L^2(\mu)} < \infty$) and $\|f\|_{L^2(\mu)} = \|g\|_{L^2(\mu)} = 1$.
We fix the parameter $\ell$ so that spt$\,f$, spt$\,g \subset [-2^{\ell}, 2^{\ell}]^{n+m}$.

We define
\begin{align*}
E_k f &= \sum_{\eta_1} \mathop{\sum_{I_1 \in \mathcal{D}_n}}_{2^{-k} < \ell(I_1) \le 2^{\ell}} h_{I_1, \eta_1} \otimes \langle f,  h_{I_1, \eta_1}\rangle_1;\\
E_k'g & = \sum_{\eta_2} \mathop{\sum_{I_2 \in \mathcal{D}'_n}}_{2^{-k} < \ell(I_2) \le 2^{\ell}} h_{I_2, \eta_2} \otimes \langle g,  h_{I_2, \eta_2}\rangle_1; \\
\tilde E_k f &= \sum_{\kappa_1} \mathop{\sum_{J_1 \in \mathcal{D}_m}}_{2^{-k} < \ell(J_1) \le 2^{\ell}} \langle f, u_{J_1, \kappa_1} \rangle_2 \otimes u_{J_1, \kappa_1}; \\
\tilde E_k' g &=  \sum_{\kappa_2} \mathop{\sum_{J_2 \in \mathcal{D}'_m}}_{2^{-k} < \ell(J_2) \le 2^{\ell}} \langle g, u_{J_2, \kappa_2} \rangle_2 \otimes u_{J_2, \kappa_2}.
\end{align*}
Here $\langle f, h_{I_1, \eta_1}\rangle_1(y_2) = \langle f(\cdot, y_2), h_{I_1, \eta_1} \rangle = \int f(y_1, y_2)h_{I_1, \eta_1}(y_1)\,d\mu_n(y_1)$.

Note that these are linear operators, $f = \lim_{k \to \infty} E_k f = \lim_{k \to \infty} \tilde E_k f$ on $L^2(\mu)$, $\|E_k f\|_{L^2(\mu)}, \|\tilde E_k f\|_{L^2(\mu)} \le \|f\|_{L^2(\mu)} = 1$ and $E_k \tilde E_k f = \tilde E_k E_k f$.
We write
\begin{align*}
\langle Tf, g\rangle &= \langle T(f-E_kf), g\rangle + \langle T(E_kf), g - E_k' g\rangle \\
&+ \langle T(E_kf - \tilde E_k E_k f), E_k' g\rangle + \langle T(\tilde E_k E_k f), E_k' g - \tilde E_k' E_k' g \rangle \\
&+ \langle T(\tilde E_k E_k f), \tilde E_k' E_k' g\rangle,
\end{align*}
or more briefly, $\langle Tf, g\rangle = \langle T(\tilde E_k E_k f), \tilde E_k' E_k' g\rangle + \epsilon_k(\omega)$, $\omega = (\omega_n, \omega_n', \omega_m, \omega_m')$.
It is easy to see using the above listed properties that $\lim_{k \to \infty} \epsilon_k(\omega) = 0$ and $|\epsilon_k(\omega)| \le 2\|T\|$. Using dominated convergence we deduce that
\begin{displaymath}
\langle Tf, g\rangle = \lim_{k \to \infty} E \langle T(\tilde E_k E_k f), \tilde E_k' E_k' g\rangle,
\end{displaymath}
where the expectation $E = E_{\omega}$.

Note that
\begin{align*}
\tilde E_k E_k f = \sum_{\eta_1, \kappa_1} \mathop{\sum_{I_1 \in \mathcal{D}_n}}_{2^{-k} < \ell(I_1) \le 2^{\ell}} \mathop{\sum_{J_1 \in \mathcal{D}_m}}_{2^{-k} < \ell(J_1) \le 2^{\ell}}
\langle f, h_{I_1, \eta_1} \otimes u_{J_1, \kappa_1} \rangle h_{I_1, \eta_1} \otimes u_{J_1, \kappa_1}.
\end{align*}
We write
\begin{align*}
\mathop{\sum_{I_1 \in \mathcal{D}_n}}_{2^{-k} < \ell(I_1) \le 2^{\ell}} \mathop{\sum_{J_1 \in \mathcal{D}_m}}_{2^{-k} < \ell(J_1) \le 2^{\ell}} =
&\mathop{\sum_{I_1 \in \mathcal{D}_{n,\, \textup{good}}}}_{2^{-k} < \ell(I_1) \le 2^{\ell}} \mathop{\sum_{J_1 \in \mathcal{D}_{m, \textup{good}}}}_{2^{-k} < \ell(J_1) \le 2^{\ell}} \\
&+ \mathop{\sum_{I_1 \in \mathcal{D}_{n,\, \textup{good}}}}_{2^{-k} < \ell(I_1) \le 2^{\ell}} \mathop{\sum_{J_1 \in \mathcal{D}_{m, \textup{bad}}}}_{2^{-k} < \ell(J_1) \le 2^{\ell}} 
+\mathop{\sum_{I_1 \in \mathcal{D}_{n,\, \textup{bad}}}}_{2^{-k} < \ell(I_1) \le 2^{\ell}} \mathop{\sum_{J_1 \in \mathcal{D}_m}}_{2^{-k} < \ell(J_1) \le 2^{\ell}}.
\end{align*}
Denoting $f_k = \tilde E_k E_k f$, this gives us the decomposition
\begin{displaymath}
f_k = f_{k, \textup{good}} + f_{k, \textup{bad}}.
\end{displaymath}
Let also $g_k = \tilde E_k' E_k' g = g_{k, \textup{good}} + g_{k, \textup{bad}}$. 

We now write
\begin{align*}
\langle Tf_k, g_k\rangle = \langle Tf_{k, \textup{good}}, g_{k, \textup{good}}\rangle + \langle Tf_{k, \textup{good}}, g_{k, \textup{bad}}\rangle +  \langle Tf_{k, \textup{bad}}, g_k\rangle.
\end{align*}
This gives us that
\begin{align*}
|\langle Tf, g\rangle| &= \lim_{k \to \infty} |E \langle Tf_k ,g_k\rangle| \\
&\le \lim_{k \to \infty} (|E \langle Tf_{k, \textup{good}}, g_{k, \textup{good}}\rangle| + \|T\| E \|g_{k, \textup{bad}}\|_{L^2(\mu)} + \|T\| E \|f_{k, \textup{bad}}\|_{L^2(\mu)}).
\end{align*}
Estimating
\begin{align*}
\|f_{k, \textup{bad}}\|_{L^2(\mu)} \le &\Big(\sum_{\eta_1, \kappa_1} \mathop{\sum_{I_1 \in \mathcal{D}_n}}_{\ell(I_1) \le 2^{\ell}}
\mathop{\sum_{J_1 \in \mathcal{D}_m}}_{\ell(J_1) \le 2^{\ell}} \chi_{\textup{bad}}(J_1) |\langle f, h_{I_1, \eta_1} \otimes u_{J_1, \kappa_1} \rangle|^2 \Big)^{1/2} \\
&+ \Big(\sum_{\eta_1, \kappa_1} \mathop{\sum_{I_1 \in \mathcal{D}_n}}_{\ell(I_1) \le 2^{\ell}} \mathop{\sum_{J_1 \in \mathcal{D}_m}}_{\ell(J_1) \le 2^{\ell}} \chi_{\textup{bad}}(I_1) |\langle f, h_{I_1, \eta_1} \otimes u_{J_1, \kappa_1} \rangle|^2 \Big)^{1/2}
\end{align*}
gives us the bound
\begin{align*}
&E\|f_{k, \textup{bad}}\|_{L^2(\mu)} \\
&\le E_{(\omega_n, \omega_n', \omega_m)} 
\Big(\sum_{\eta_1, \kappa_1} \mathop{\sum_{I_1 \in \mathcal{D}_n}}_{\ell(I_1) \le 2^{\ell}} \mathop{\sum_{J_1 \in \mathcal{D}_m}}_{\ell(J_1) \le 2^{\ell}} \mathbb{P}_{\omega_m'}(J_1 \textup{ bad})
 |\langle f, h_{I_1, \eta_1} \otimes u_{J_1, \kappa_1} \rangle|^2 \Big)^{1/2} \\
&+ E_{(\omega_n, \omega_m, \omega_m')}  \Big(\sum_{\eta_1, \kappa_1}
\mathop{\sum_{I_1 \in \mathcal{D}_n}}_{\ell(I_1) \le 2^{\ell}} \mathop{\sum_{J_1 \in \mathcal{D}_m}}_{\ell(J_1) \le 2^{\ell}} \mathbb{P}_{\omega_n'}(I_1 \textup{ bad}) |\langle f, h_{I_1, \eta_1} \otimes u_{J_1, \kappa_1} \rangle|^2 \Big)^{1/2} \le c(r),
\end{align*}
where $c(r) \to 0$, when $r \to \infty$. Fixing $r$ to be sufficiently large we have shown that
\begin{equation}\label{eq:InitialReduction}
  0.7\|T\| \le |\langle Tf, g\rangle| \le  \lim_{k \to \infty} |E\langle Tf_{k, \textup{good}}, g_{k, \textup{good}}\rangle| + 0.1\|T\|.
\end{equation}

\subsection{Outline of the core of the proof}
The core of the proof, which will span most of the rest of the paper, consists of showing that
\begin{equation}\label{eq:CoreToProve}
\begin{split}
&|E\langle Tf_{k, \textup{good}}, g_{k, \textup{good}}\rangle| \\
&= \Big| E \mathop{\sum_{\eta_1, \eta_2}}_{\kappa_1, \kappa_2} \mathop{\sum_{I_1, I_2 \, \textup{good}}}_{2^{-k} < \ell(I_1),\, \ell(I_2) \le 2^{\ell}} \mathop{\sum_{J_1, J_2\, \textup{good}}}_{2^{-k} < \ell(J_1),\, \ell(J_2) \le 2^{\ell}}
\langle f, h_{I_1, \eta_1} \otimes u_{J_1, \kappa_1} \rangle \\
&\hspace{3.5cm} \times \langle g, h_{I_2, \eta_2} \otimes u_{J_2, \kappa_2} \rangle \langle T(h_{I_1, \eta_1} \otimes u_{J_1, \kappa_1}), h_{I_2, \eta_2} \otimes u_{J_2, \kappa_2}\rangle\Big| \\
&\leq (C+0.1\|T\|)\|f\|_{L^2(\mu)}\|g\|_{L^2(\mu)}=C+0.1\|T\|,
\end{split}
\end{equation}
and hence, combining with \eqref{eq:InitialReduction},
\begin{equation*}
\begin{split}
  0.6\|T\| &\leq C+0.1\|T\|,\qquad\textup{which implies}\\
  \|T\| &\leq 2C,
\end{split}
\end{equation*}
completing the proof of Theorem~\ref{thm:main}.

For the proof of \eqref{eq:CoreToProve}, we fix $k$ (recall that $\ell$ is already fixed). We will bound $|E\langle Tf_{k, \textup{good}}, g_{k, \textup{good}}\rangle|$ uniformly on these quantities.
We suppress the restrictions
\begin{displaymath}
2^{-k} < \ell(I_1),\, \ell(I_2), \ell(J_1), \ell(J_2) \le 2^{\ell}
\end{displaymath}
from the notation. The goodness of the cubes is essential and so is the fact that we have the averaging operator $E$ in front (this is used in the proof to control multiple bad boundary regions).
However, sometimes (for the collapse of paraproducts) the fact that all the cubes are good is a problem. We deal with this during the proof by sometimes explicitly
replacing $f$ by $f_{\textup{good}}$ noting that $\langle f_{\textup{good}}, h_{I_1, \eta_1} \otimes u_{J_1, \kappa_1}\rangle = \langle f, h_{I_1, \eta_1} \otimes u_{J_1, \kappa_1}\rangle$ if
$I_1$ and $J_1$ are good, and $\langle f_{\textup{good}}, h_{I_1, \eta_1} \otimes u_{J_1, \kappa_1}\rangle = 0$ otherwise.

We will focus on the part of the summation, where $\ell(I_1) \le \ell(I_2)$ and $\ell(J_1) \le \ell(J_2)$. One exception of this is made during the proof to handle a certain mixed full paraproduct appearing in the summation
 $\ell(I_1) \le \ell(I_2)$, $\ell(J_1) > \ell(J_2)$. It has an essential difference to the full paraproduct appearing in our main summation, so it does need separate attention.
 
 In any case, we perform the splitting
 \begin{displaymath}
 \sum_{\ell(I_1) \le \ell(I_2)} =  \mathop{\sum_{\ell(I_1) \le \ell(I_2)}}_{d(I_1, I_2) > 2\ell(I_1)^{\gamma_n}\ell(I_2)^{1-\gamma_n}} +  \mathop{\sum_{\ell(I_1) < 2^{-r}\ell(I_2)}}_{d(I_1, I_2) \le 2\ell(I_1)^{\gamma_n}\ell(I_2)^{1-\gamma_n}}
 + \mathop{\sum_{2^{-r}\ell(I_2) \le \ell(I_1) \le \ell(I_2)}}_{d(I_1, I_2) \le 2\ell(I_1)^{\gamma_n}\ell(I_2)^{1-\gamma_n}}.
 \end{displaymath}
 These three parts are called separated, nested and adjacent respectively. The term nested makes sense, since the summing conditions (recalling that $I_1$ is good)
 actually imply that there is a child $I_{2,1} \in \textup{ch}(I_2)$ so that $d(I_1, I_{2,1}^c) > 4\ell(I_1)^{\gamma_n}\ell(I_{2,1})^{1-\gamma_n}$.
 
A similar splitting in the summation $\ell(J_1) \le \ell(J_2)$ is also performed. This splits the whole summation into nine parts. We explicitly deal with the following cases  (the remaining three being symmetric to one of these):

\begin{tabular}{l|ccc}
\qquad$(J_1,J_2)$ & separated & adjacent & nested \\ 
 $(I_1,I_2)$\qquad\mbox{} & & & \\ \hline
 separated & $*$ & $*$ & $*$ \\
 adjacent & & $*$ & $*$ \\
 nested & & & $*$ \\
\end{tabular}

The cases where the first pair $(I_1,I_2)$ is separated are treated in Section~\ref{sec:sep}, and the cases where it is adjacent in Section~\ref{sec:adj}. The remaining cases where both pairs are nested are handled in Sections~\ref{sec:nest} and \ref{sec:mixed}. A combination of these gives the core estimate \eqref{eq:CoreToProve} and completes the proof of Theorem~\ref{thm:main}. In the final Section~\ref{sec:bmo}, we investigate the necessity of the condition that $T1\in\textup{BMO}_{\textup{prod}}(\mu)$.

\begin{rem}[Vinogradov notation and implicit constants]
We will use the notation $f \lesssim g$ synonymously with $f \le Cg$ for some constant $C$. We also use $f \sim g$ if $f \lesssim g \lesssim f$. The dependence on the various parameters should be somewhat clear, but basically $C$ may depend on the various constants involved in the assumptions.
\end{rem}

\section{Separated cubes}\label{sec:sep}
 
In this section we deal with the part of the summation, where at least the pair of cubes $(I_1,I_2)$ is separated. This further splits into subcases according to the nature of the other pair $(J_1,J_2)$.
 
\subsection{Separated/separated}
\begin{lem}\label{ccss}
Let $I_1 \in \mathcal{D}_n$, $I_2 \in \mathcal{D}_n'$, $J_1 \in \mathcal{D}_m$ and $J_2 \in \mathcal{D}_m'$ be such that
$\ell(I_1) \le \ell(I_2)$, $\ell(J_1) \le \ell(J_2)$, $d(I_1, I_2) > 2\ell(I_1)^{\gamma_n}\ell(I_2)^{1-\gamma_n}$ and $d(J_1, J_2) > 2\ell(J_1)^{\gamma_m}\ell(J_2)^{1-\gamma_m}$. Let
$f_{I_1}$, $f_{J_1}$, $g_{I_2}$ and $g_{J_2}$ be functions supported on $I_1$, $J_1$, $I_2$ and $J_2$ respectively, and assume that
$\|f_{I_1}\|_{L^2(\mu_n)} = \|f_{J_1}\|_{L^2(\mu_m)} = \|g_{I_2}\|_{L^2(\mu_n)}  = \|g_{J_2}\|_{L^2(\mu_m)} = 1$ and 
 $\int f_{I_1}\,d\mu_n = \int f_{J_1}\,d\mu_m = 0$.
Then there holds that
\begin{align*}
|\langle T(f_{I_1} \otimes f_{J_1}), g_{I_2} \otimes g_{J_2}\rangle| \lesssim A^{\textup{sep}}_{I_1I_2}A^{\textup{sep}}_{J_1J_2}, 
\end{align*}
where
\begin{displaymath}
A^{\textup{sep}}_{I_1I_2} = \frac{\ell(I_1)^{\alpha/2}\ell(I_2)^{\alpha/2}}{D(I_1,I_2)^{\alpha}\sup_{z \in I_1} \lambda_n(z, D(I_1,I_2))} \mu_n(I_1)^{1/2} \mu_n(I_2)^{1/2}
\end{displaymath}
and
\begin{displaymath}
A^{\textup{sep}}_{J_1J_2} = \frac{\ell(J_1)^{\beta/2}\ell(J_2)^{\beta/2}}{D(J_1,J_2)^{\beta}\sup_{w \in J_1} \lambda_m(w, D(J_1,J_2))} \mu_m(J_1)^{1/2} \mu_m(J_2)^{1/2}.
\end{displaymath}
Here $D(I_1, I_2) = \ell(I_1) + \ell(I_2) + d(I_1,I_2)$.
\end{lem}
\begin{proof}
We write $\langle T(f_{I_1} \otimes f_{J_1}), g_{I_2} \otimes g_{J_2}\rangle$ as the integral
\begin{displaymath}
\int_{I_1} \int_{J_1} \int_{I_2} \int_{J_2} K(x,y) f_{I_1}(y_1) f_{J_1}(y_2)  g_{I_2}(x_1)  g_{J_2}(x_2) \,d\mu_m(x_2)\,d\mu_n(x_1)\,d\mu_m(y_2) \, d\mu_n(y_1)
\end{displaymath}
and then replace $K(x,y)$ by
\begin{displaymath}
K(x,y) - K(x, (y_1, w)) - K(x, (z, y_2)) + K(x, (z,w)),
\end{displaymath}
where $z \in I_1$ and $w \in J_1$ are arbitrary. The replacement can be done inside the integral, since  $\int f_{I_1}\,d\mu_n = \int f_{J_1}\,d\mu_m = 0$.

Next, one notes that
\begin{displaymath}
|y_1 - z| \le \ell(I_1) \le \ell(I_1)^{\gamma_n}\ell(I_2)^{1-\gamma_n} \le d(I_1,I_2)/2 \le |x_1-z|/2,
\end{displaymath}
and that similarly $|y_2-w| \le |x_2-w|/2$. Using the H\"older condition of $K$ this yields
\begin{align*}
&|K(x,y) - K(x, (y_1, w)) - K(x, (z, y_2)) + K(x, (z,w))|  \\
&\lesssim \frac{|y_1-z|^{\alpha}}{|x_1-z|^{\alpha} \lambda_n(z, |x_1-z|)} \frac{|y_2-w|^{\beta}}{|x_2-w|^{\beta} \lambda_m(w, |x_2-w|)}  \\
&\le \frac{\ell(I_1)^{\alpha}}{d(I_1,I_2)^{\alpha} \lambda_n(z, d(I_1,I_2))} \frac{\ell(J_1)^{\beta}}{d(J_1,J_2)^{\beta} \lambda_m(w, d(J_1,J_2))}. 
\end{align*}
Splitting into cases $d(I_1,I_2) \ge \ell(I_2)$ and $d(I_1,I_2) < \ell(I_2)$ one can show precisely like in the proof of \cite[Lemma 6.2]{HM} that
\begin{displaymath}
\frac{\ell(I_1)^{\alpha}}{d(I_1,I_2)^{\alpha} \lambda_n(z, d(I_1,I_2))} \lesssim \frac{\ell(I_1)^{\alpha/2}\ell(I_2)^{\alpha/2}}{D(I_1,I_2)^{\alpha} \lambda_n(z, D(I_1,I_2))}.
\end{displaymath}
The result of the lemma readily follows.
\end{proof}

We recall \cite[Proposition 6.3]{HM}:
\begin{prop}\label{l2s}
There holds
\begin{displaymath}
\sum_{\ell(I_1) \le \ell(I_2)} A^{\textup{sep}}_{I_1I_2} x_{I_1} y_{I_2} \lesssim \Big( \sum_{I_1} x_{I_1}^2 \Big)^{1/2}  \Big( \sum_{I_2} y_{I_2}^2 \Big)^{1/2}
\end{displaymath}
for $x_{I_1}, y_{I_2} \ge 0$. In particular, there holds that
\begin{displaymath}
\Big( \sum_{I_2} \Big[ \sum_{I_1:\,\ell(I_1) \le \ell(I_2)} A^{\textup{sep}}_{I_1I_2} x_{I_1} \Big]^2\Big)^{1/2} \lesssim \Big( \sum_{I_1} x_{I_1}^2 \Big)^{1/2}.
\end{displaymath}
\end{prop}
We are to bound
\begin{align*}
\mathop{\sum_{\eta_1, \eta_2}}_{\kappa_1, \kappa_2}   \mathop{\sum_{\ell(I_1) \le  \ell(I_2)}}_{d(I_1, I_2) > 2\ell(I_1)^{\gamma_n}\ell(I_2)^{1-\gamma_n}} &
 \mathop{\sum_{\ell(J_1) \le \ell(J_2)}}_{d(J_1, J_2) > 2\ell(J_1)^{\gamma_m}\ell(J_2)^{1-\gamma_m}} |\langle f, h_{I_1, \eta_1} \otimes u_{J_1, \kappa_1}\rangle|\\
 &\times  |\langle g, h_{I_2, \eta_2} \otimes u_{J_2, \kappa_2}\rangle|
 |\langle T(h_{I_1, \eta_1} \otimes u_{J_1, \kappa_1}),  h_{I_2, \eta_2} \otimes u_{J_2, \kappa_2}\rangle|.
\end{align*}
In general, we need to split
\begin{displaymath}
\mathop{\sum_{\eta_1, \eta_2}}_{\kappa_1, \kappa_2} = \mathop{\mathop{\mathop{\sum_{\eta_1 = 0}}_{\kappa_1 = 0}}}_{\eta_2, \kappa_2} + \mathop{\mathop{\mathop{\sum_{\eta_1 = 0}}_{\kappa_1 \ne 0}}}_{\eta_2, \kappa_2}
+ \mathop{\mathop{\mathop{\sum_{\eta_1 \ne 0}}_{\kappa_1 = 0}}}_{\eta_2, \kappa_2} + \mathop{\mathop{\mathop{\sum_{\eta_1 \ne 0}}_{\kappa_1 \ne 0}}}_{\eta_2, \kappa_2}.
\end{displaymath}
However, in this separated/separated sum only the fourth summation appears. Indeed, say $\eta_1 = 0$. Then $2^{\ell} = \ell(I_1) \le \ell(I_2) \le 2^{\ell}$ i.e. $\ell(I_1) = \ell(I_2) = 2^{\ell}$.
If the cubes $I_1$ and $I_2$ actually contribute to the sum, one must have $I_1 \cap [-2^{\ell}, 2^{\ell}]^n \ne \emptyset$ and $I_2  \cap [-2^{\ell}, 2^{\ell}]^n \ne \emptyset$. But this means that
$d(I_1, I_2) \le 2^{\ell+1} = 2\ell(I_1)^{\gamma_n}\ell(I_2)^{1-\gamma_n}$.

The main term, where $\eta_1 \ne 0$ and $\kappa_1 \ne 0$, is the only one that remains and can now be handled. In that sum Lemma \ref{ccss} gives
\begin{displaymath}
|\langle T(h_{I_1, \eta_1} \otimes u_{J_1, \kappa_1}),  h_{I_2, \eta_2} \otimes u_{J_2, \kappa_2}\rangle|  \lesssim A^{\textup{sep}}_{I_1I_2}A^{\textup{sep}}_{J_1J_2}.
\end{displaymath}
With fixed  $\eta_1 \ne 0$, $\kappa_1 \ne 0$, $\eta_2$ and $\kappa_2$ we get from Proposition \ref{l2s} that
\begin{align*}
& \sum_{\ell(I_1) \le \ell(I_2)} \sum_{\ell(J_1) \le \ell(J_2)} A^{\textup{sep}}_{I_1I_2}A^{\textup{sep}}_{J_1J_2} |\langle f, h_{I_1, \eta_1} \otimes u_{J_1, \kappa_1}\rangle| |\langle g, h_{I_2, \eta_2} \otimes u_{J_2, \kappa_2}\rangle| \\
&\lesssim \sum_{\ell(I_1) \le \ell(I_2)} A^{\textup{sep}}_{I_1I_2} \Big( \sum_{J_1}  |\langle f, h_{I_1, \eta_1} \otimes u_{J_1, \kappa_1}\rangle|^2 \Big)^{1/2} \Big(\sum_{J_2}  |\langle g, h_{I_2, \eta_2} \otimes u_{J_2, \kappa_2}\rangle|^2 \Big)^{1/2} \\
&\lesssim \Big(\sum_{I_1} \sum_{J_1}  |\langle f, h_{I_1, \eta_1} \otimes u_{J_1, \kappa_1}\rangle|^2 \Big)^{1/2} \Big( \sum_{I_2} \sum_{J_2} |\langle g, h_{I_2, \eta_2} \otimes u_{J_2, \kappa_2}\rangle|^2 \Big)^{1/2} \\
&\lesssim \|f\|_{L^2(\mu)} \|g\|_{L^2(\mu)}.
\end{align*}

\subsection{Separated/nested}
We begin by remarking that in this sum we automatically have $\eta_1 \ne 0$ and $\kappa_1 \ne 0$.
We write $\langle T(h_{I_1, \eta_1} \otimes u_{J_1, \kappa_1}),  h_{I_2, \eta_2} \otimes u_{J_2, \kappa_2}\rangle$ as the sum
\begin{align*}
[\langle T(h_{I_1, \eta_1} \otimes u_{J_1, \kappa_1}),  h_{I_2, \eta_2} \otimes u_{J_2, \kappa_2}\rangle
- \langle &u_{J_2,\kappa_2} \rangle_{J_1} \langle T(h_{I_1, \eta_1} \otimes u_{J_1, \kappa_1}),  h_{I_2, \eta_2} \otimes 1\rangle] \\
&+ \langle u_{J_2,\kappa_2} \rangle_{J_1} \langle T(h_{I_1, \eta_1} \otimes u_{J_1, \kappa_1}),  h_{I_2, \eta_2} \otimes 1\rangle.
\end{align*}
We start by dealing with the sum having the first term as the matrix element. Let $J_{2,1} \in \textup{ch}(J_2)$ be such that $J_1 \subset J_{2,1}$.
Note that $\langle u_{J_2,\kappa_2} \rangle_{J_1} = \langle u_{J_2,\kappa_2} \rangle_{J_{2,1}}$. Therefore, we have that
\begin{displaymath}
\langle T(h_{I_1, \eta_1} \otimes u_{J_1, \kappa_1}),  h_{I_2, \eta_2} \otimes u_{J_2, \kappa_2}\rangle
- \langle u_{J_2,\kappa_2} \rangle_{J_1} \langle T(h_{I_1, \eta_1} \otimes u_{J_1, \kappa_1}),  h_{I_2, \eta_2} \otimes 1\rangle
\end{displaymath}
equals
\begin{align*}
-\langle u_{J_2,\kappa_2} \rangle_{J_{2,1}}  \langle T(h_{I_1, \eta_1}& \otimes u_{J_1, \kappa_1}),  h_{I_2, \eta_2} \otimes \chi_{J_{2,1}^c}\rangle \\
&+ \mathop{\sum_{J_2' \in \textup{ch}(J_2)}}_{J_2' \subset J_2 \setminus J_{2,1}} \langle T(h_{I_1, \eta_1} \otimes u_{J_1, \kappa_1}),  h_{I_2, \eta_2} \otimes u_{J_2, \kappa_2}\chi_{J_2'}\rangle.
\end{align*}
As $J_1$ is good, $J_1 \subset J_{2,1}$ and $\ell(J_1) \le 2^{-r}\ell(J_{2,1})$, we have
\begin{displaymath}
d(J_1, J_{2,1}^c) \ge 4\ell(J_1)^{\gamma_m}\ell(J_{2,1})^{1-\gamma_m} > 2\ell(J_1)^{\gamma_m}\ell(J_2)^{1-\gamma_m}.
\end{displaymath}
\begin{lem}\label{mixedsep1}
There holds
\begin{displaymath}
|\langle u_{J_2,\kappa_2} \rangle_{J_{2,1}}  \langle T(h_{I_1, \eta_1} \otimes u_{J_1, \kappa_1}),  h_{I_2, \eta_2} \otimes \chi_{J_{2,1}^c}\rangle| \lesssim
A_{I_1I_2}^{\textup{sep}}A_{J_1J_2}^{\textup{in}},
\end{displaymath}
where
\begin{displaymath}
A_{J_1J_2}^{\textup{in}} = \Big( \frac{\ell(J_1)}{\ell(J_2)}\Big)^{\beta/2} \Big( \frac{\mu_m(J_1)}{\mu_m(J_{2,1})}\Big)^{1/2}.
\end{displaymath}
\end{lem}
\begin{proof}
We write $\langle T(h_{I_1, \eta_1} \otimes u_{J_1, \kappa_1}),  h_{I_2, \eta_2} \otimes \chi_{J_{2,1}^c}\rangle$ as
the integral
\begin{displaymath}
\int_{I_1} \int_{J_1} \int_{I_2} \int_{J_{2,1}^c} K(x,y)h_{I_1, \eta_1}(y_1) u_{J_1, \kappa_1}(y_2)  h_{I_2, \eta_2}(x_1) \,d\mu_m(x_2)\,d\mu_n(x_1)\,d\mu_m(y_2) \, d\mu_n(y_1),
\end{displaymath}
and then, using the fact that $\eta_1 \ne 0$ and $\kappa_1 \ne 0$, replace $K(x,y)$ with
\begin{displaymath}
K(x,y) - K(x, (y_1, w)) - K(x, (z, y_2)) + K(x, (z,w)).
\end{displaymath}
Here $z \in I_1$ and $w \in J_1$ are arbitrary. Since $d(I_1, I_2) > 2\ell(I_1)^{\gamma_n}\ell(I_2)^{1-\gamma_n}$ and $d(J_1, J_{2,1}^c) > 2\ell(J_1)^{\gamma_m}\ell(J_2)^{1-\gamma_m}$, we have
$|y_1 - z| \le |x_1-z|/2$ and $|y_2-w| \le |x_2-w|/2$. Now, we have
\begin{align*}
|\langle T(h_{I_1, \eta_1} \otimes u_{J_1, \kappa_1}),  h_{I_2, \eta_2} \otimes \chi_{J_{2,1}^c}\rangle| \lesssim A&_{I_1I_2}^{\textup{sep}} 
\ell(J_1)^{\beta} \mu_m(J_1)^{1/2} \\ &\times \int_{\R^m \setminus B(w, d(J_1,J_{2,1}^c))} \frac{|x_2-w|^{-\beta}}{\lambda_m(w, |x_2-w|)}\,d\mu_m(x_2),
\end{align*}
where
\begin{displaymath}
\int_{\R^m \setminus B(w, d(J_1,J_{2,1}^c))} \frac{|x_2-w|^{-\beta}}{\lambda_m(w, |x_2-w|)}\,d\mu_m(x_2) \lesssim d(J_1,J_{2,1}^c)^{-\beta} \lesssim \ell(J_1)^{-\beta/2}\ell(J_2)^{-\beta/2}.
\end{displaymath}
Noting that $|\langle u_{J_2,\kappa_2} \rangle_{J_{2,1}}| \le \mu_m(J_{2,1})^{-1/2}$ finishes the proof.
\end{proof}
\begin{lem}\label{mixedsep2}
If $J_2' \in \textup{ch}(J_2)$ and $J_2' \subset J_2 \setminus J_{2,1}$, then
\begin{displaymath}
|\langle T(h_{I_1, \eta_1} \otimes u_{J_1, \kappa_1}),  h_{I_2, \eta_2} \otimes u_{J_2, \kappa_2}\chi_{J_2'}\rangle| \lesssim A_{I_1I_2}^{\textup{sep}}A_{J_1J_2}^{\textup{in}}.
\end{displaymath}
\end{lem}
\begin{proof}
Since $d(J_1, J_2') \ge d(J_1, J_{2,1}^c) > 2\ell(J_1)^{\gamma_m}\ell(J_2')^{1-\gamma_m}$, Lemma \ref{ccss} gives that
\begin{displaymath}
|\langle T(h_{I_1, \eta_1} \otimes u_{J_1, \kappa_1}),  h_{I_2, \eta_2} \otimes u_{J_2, \kappa_2}\chi_{J_2'}\rangle| \lesssim A_{I_1I_2}^{\textup{sep}}
\Big( \frac{\ell(J_1)}{\ell(J_2)}\Big)^{\beta/2} \mu_m(J_1)^{1/2} \frac{\mu_m(J_2')^{1/2}}{\lambda_m(w, \ell(J_2'))}
\end{displaymath}
with any $w \in J_1 \subset J_{2,1}$. It is easy to see that
\begin{displaymath}
\frac{\mu_m(J_2')^{1/2}}{\lambda_m(w, \ell(J_2'))} \lesssim \mu_m(J_{2,1})^{-1/2},
\end{displaymath}
which ends the proof of the lemma.
\end{proof}
We recall \cite[Lemma 7.4]{NTV}:

\begin{lem}\label{InSummation}
There holds
\begin{displaymath}
\mathop{\mathop{\sum_{J_1,\,J_2}}_{J_1 \subset J_{2,1} \in\, \textup{ch}(J_2)}}_{\ell(J_1) < 2^{-r}\ell(J_2)} A_{J_1J_2}^{\textup{in}} x_{J_1} y_{J_2} \lesssim  \Big( \sum_{J_1} x_{J_1}^2 \Big)^{1/2}  \Big( \sum_{J_2} y_{J_2}^2 \Big)^{1/2}
\end{displaymath}
for $x_{J_1}, y_{J_2} \ge 0$. In particular, there holds that
\begin{displaymath}
\Big( \sum_{J_2} \Big[ \mathop{\mathop{\sum_{J_1}}_{J_1 \subset J_{2,1} \in\, \textup{ch}(J_2)}}_{\ell(J_1) < 2^{-r}\ell(J_2)} A_{J_1J_2}^{\textup{in}} x_{J_1} \Big]^2\Big)^{1/2}  \lesssim \Big( \sum_{J_1} x_{J_1}^2 \Big)^{1/2}.
\end{displaymath}
\end{lem}

Combining this lemma with Proposition \ref{l2s} and the above estimates gives that
the summation
\begin{displaymath}
\mathop{\sum_{\ell(I_1) \le  \ell(I_2)}}_{d(I_1, I_2) > 2\ell(I_1)^{\gamma_n}\ell(I_2)^{1-\gamma_n}}
\mathop{\sum_{\ell(J_1) <  2^{-r}\ell(J_2)}}_{d(J_1, J_2) \le 2\ell(J_1)^{\gamma_m}\ell(J_2)^{1-\gamma_m}} 
\end{displaymath}
is dominated by $\|f\|_{L^2(\mu)} \|g\|_{L^2(\mu)}$, if we are using the modified matrix element
\begin{displaymath}
\langle T(h_{I_1, \eta_1} \otimes u_{J_1, \kappa_1}),  h_{I_2, \eta_2} \otimes u_{J_2, \kappa_2}\rangle
- \langle u_{J_2,\kappa_2} \rangle_{J_1} \langle T(h_{I_1, \eta_1} \otimes u_{J_1, \kappa_1}),  h_{I_2, \eta_2} \otimes 1\rangle.
\end{displaymath}

What is left to do is to consider the same summation but with the matrix element
\begin{displaymath}
\langle u_{J_2,\kappa_2} \rangle_{J_1} \langle T(h_{I_1, \eta_1} \otimes u_{J_1, \kappa_1}),  h_{I_2, \eta_2} \otimes 1\rangle.
\end{displaymath}
Recall that all the appearing cubes are good. In particular, $J_2$ is good and this is actually slightly problematic. However, for good $I_2$ and $J_2$ we can write
$\langle g, h_{I_2, \eta_2} \otimes u_{J_2, \kappa_2}\rangle = \langle g_{\textup{good}}, h_{I_2, \eta_2} \otimes u_{J_2, \kappa_2}\rangle$, and after this
we may add all the bad $J_2$ to the summation, since $\langle g_{\textup{good}}, h_{I_2, \eta_2} \otimes u_{J_2, \kappa_2}\rangle = 0$ if $J_2$ is bad. Moreover, in the non-trivial
case $\langle u_{J_2,\kappa_2} \rangle_{J_1} \ne 0$, we have $J_1 \cap J_2 \ne \emptyset$ so that we may remove the summing condition
$d(J_1, J_2) \le 2\ell(J_1)^{\gamma_m}\ell(J_2)^{1-\gamma_m}$.

In what follows we explicitly write the fact that $I_1, I_2$ and $J_1$ are good, since we need to highlight the fact that $J_2$ is not. 
For the $J_2$ summation we also write out the hidden summing condition $\ell(J_2) \le 2^{\ell}$.
For fixed $\eta_1, \kappa_1$ and $\eta_2$ we need to bound
\begin{align*}
\mathop{\mathop{\sum_{I_1, I_2 \textup{ good}}}_{\ell(I_1) \le  \ell(I_2)}}_{d(I_1, I_2) > 2\ell(I_1)^{\gamma_n}\ell(I_2)^{1-\gamma_n}} \sum_{J_1 \textup{ good}} 
\Big\langle &\mathop{\sum_{J_2}}_{2^r\ell(J_1) < \ell(J_2) \le 2^{\ell}}  \sum_{\kappa_2}  \langle g_{\textup{good}}, h_{I_2, \eta_2} \otimes u_{J_2, \kappa_2}\rangle u_{J_2, \kappa_2} \Big\rangle_{J_1} \\
&\times \langle f, h_{I_1, \eta_1} \otimes u_{J_1, \kappa_1}\rangle
\langle T(h_{I_1, \eta_1} \otimes u_{J_1, \kappa_1}),  h_{I_2, \eta_2} \otimes 1\rangle.
\end{align*}
We write $\langle g_{\textup{good}}, h_{I_2, \eta_2} \otimes u_{J_2, \kappa_2}\rangle = \langle g_{\textup{good}}^{I_2,\eta_2}, u_{J_2, \kappa_2}\rangle$, where
$g_{\textup{good}}^{I_2,\eta_2}(y) = \langle g_{\textup{good}}, h_{I_2, \eta_2}\rangle_1(y) = \int_{\R^n} g_{\textup{good}}(x,y) h_{I_2, \eta_2}(x)\,d\mu_n(x)$. Then we note that
\begin{displaymath}
\sum_{\kappa_2}  \langle g_{\textup{good}}^{I_2,\eta_2}, u_{J_2, \kappa_2}\rangle u_{J_2, \kappa_2}(y) 
= \left\{ \begin{array}{ll}
\Delta_{J_2}(g_{\textup{good}}^{I_2,\eta_2})(y), & \textup{if } \ell(J_2) < 2^{\ell}, \\
\Delta_{J_2}(g_{\textup{good}}^{I_2,\eta_2})(y) + E_{J_2}(g_{\textup{good}}^{I_2,\eta_2})(y), & \textup{if } \ell(J_2) = 2^{\ell}.
\end{array} \right.
\end{displaymath}
Therefore, the sum inside the average over $J_1$ equals the constant $\langle g_{\textup{good}}^{I_2,\eta_2} \rangle_{S(J_1)}$, where $S(J_1) \in \mathcal{D}_m'$ is the unique
cube for which $\ell(S(J_1)) = 2^r\ell(J_1)$ and $J_1 \subset S(J_1)$. Note that for this collapse it is important that $J_2$ is not restricted to good cubes. The cube $S(J_1)$ exists
since $J_1$ is good.

We are left to bound
\begin{align*}
\sum_{\eta_1, \kappa_1, \eta_2} \mathop{\mathop{\sum_{I_1, I_2 \textup{ good}}}_{\ell(I_1) \le  \ell(I_2)}}_{d(I_1, I_2) > 2\ell(I_1)^{\gamma_n}\ell(I_2)^{1-\gamma_n}} \sum_{J_2} 
\mathop{\mathop{\sum_{J_1 \textup{ good}}}_{J_1 \subset J_2}}_{\ell(J_1) = 2^{-r}\ell(J_2)} \langle g_{\textup{good}}^{I_2,\eta_2} \rangle_{J_2}\langle T(h_{I_1, \eta_1} \otimes& u_{J_1, \kappa_1}),  h_{I_2, \eta_2} \otimes 1\rangle \\
&\times \langle f, h_{I_1, \eta_1} \otimes u_{J_1, \kappa_1}\rangle.
\end{align*}
We note that this equals
\begin{displaymath}
\Big\langle \sum_{\eta_1, \kappa_1, \eta_2} \mathop{\mathop{\sum_{I_1, I_2 \textup{ good}}}_{\ell(I_1) \le  \ell(I_2)}}_{d(I_1, I_2) > 2\ell(I_1)^{\gamma_n}\ell(I_2)^{1-\gamma_n}}
h_{I_2, \eta_2} \otimes (\Pi^{\kappa_1}_{b_{I_1I_2}^{\eta_1\eta_2}})^*f_{I_1}^{\eta_1}, g_{\textup{good}}\Big\rangle,
\end{displaymath}
where $f_{I_1}^{\eta_1} = \langle f, h_{I_1,\eta_1} \rangle_1$, $b_{I_1I_2}^{\eta_1\eta_2} = \langle T^*(h_{I_2,\eta_2} \otimes 1), h_{I_1,\eta_1}\rangle_1$, and
\begin{displaymath}
\Pi^{\kappa_1}_a w = \sum_{J_2 \in \mathcal{D}_m'} \mathop{\mathop{\sum_{J_1 \in \mathcal{D}_{m, \textup{good}}}}_{J_1 \subset J_2}}_{\ell(J_1) = 2^{-r}\ell(J_2)} \langle w \rangle_{J_2} \langle a, u_{J_1,\kappa_1} \rangle u_{J_1,\kappa_1}, \qquad \kappa_1 \ne 0.
\end{displaymath}
Here, of course, the pairings and averages are taken with respect to the $\mu_m$ measure.

There holds
\begin{align*}
\Big\|&\Big\langle \sum_{\eta_1, \kappa_1, \eta_2} \mathop{\mathop{\sum_{I_1, I_2 \textup{ good}}}_{\ell(I_1) \le  \ell(I_2)}}_{d(I_1, I_2) > 2\ell(I_1)^{\gamma_n}\ell(I_2)^{1-\gamma_n}}
h_{I_2, \eta_2} \otimes (\Pi^{\kappa_1}_{b_{I_1I_2}^{\eta_1\eta_2}})^*f_{I_1}^{\eta_1}, g_{\textup{good}}\Big\rangle\Big\|_{L^2(\mu)} \\
&\le \sum_{\eta_1, \kappa_1, \eta_2} \Big\|  \mathop{\mathop{\sum_{I_1, I_2 \textup{ good}}}_{\ell(I_1) \le  \ell(I_2)}}_{d(I_1, I_2) > 2\ell(I_1)^{\gamma_n}\ell(I_2)^{1-\gamma_n}}
h_{I_2, \eta_2} \otimes (\Pi^{\kappa_1}_{b_{I_1I_2}^{\eta_1\eta_2}})^*f_{I_1}^{\eta_1}\Big\|_{L^2(\mu)} \|g\|_{L^2(\mu)} \\
 &= \sum_{\eta_1, \kappa_1, \eta_2} \Big(\sum_{I_2 \textup{ good}}\Big\|  \mathop{\mathop{\sum_{I_1\textup{ good}}}_{\ell(I_1) \le  \ell(I_2)}}_{d(I_1, I_2) > 2\ell(I_1)^{\gamma_n}\ell(I_2)^{1-\gamma_n}}
   (\Pi^{\kappa_1}_{b_{I_1I_2}^{\eta_1\eta_2}})^*f_{I_1}^{\eta_1}\Big\|_{L^2(\mu_m)}^2\Big)^{1/2} \|g\|_{L^2(\mu)} \\
 &\lesssim \sum_{\eta_1, \kappa_1, \eta_2} \Big( \sum_{I_2}\Big[  \mathop{\sum_{I_1:\, \ell(I_1) \le  \ell(I_2)}}_{d(I_1, I_2) > 2\ell(I_1)^{\gamma_n}\ell(I_2)^{1-\gamma_n}}
   \|b_{I_1I_2}^{\eta_1\eta_2}\|_{\textup{BMO}^2_3(\mu_m)} \|f_{I_1}^{\eta_1}\|_{L^2(\mu_m)}\Big]^2\Big)^{1/2} \|g\|_{L^2(\mu)},
\end{align*}
where we used the orthonormality of the functions $h_{I_2,\eta_2}\in L^2(\mu_n)$, and the following paraproduct boundedness result:

\begin{lem}
For any $M>1$, we have
\begin{equation*}
  \|\Pi^{\kappa}_a\|_{L^2(\mu_m)\to L^2(\mu_m)}
  \lesssim\|a\|_{\textup{BMO}^2_M(\mu_m)}.
\end{equation*}
\end{lem}

%

The BMO norms above are estimated as follows:
\begin{lem}
If $\eta_1 \ne 0$, $\ell(I_1) \le \ell(I_2)$ and $d(I_1, I_2) > 2\ell(I_1)^{\gamma_n}\ell(I_2)^{1-\gamma_n}$, then there holds that
\begin{displaymath}
\|b_{I_1I_2}^{\eta_1\eta_2}\|_{\textup{BMO}^2_{3}(\mu_m)} \lesssim A_{I_1I_2}^{\textup{sep}}
\end{displaymath}
for $b_{I_1I_2}^{\eta_1\eta_2} = \langle T^*(h_{I_2,\eta_2} \otimes 1), h_{I_1,\eta_1}\rangle_1$.
\end{lem}

\begin{proof}
We fix a cube $V \subset \R^m$ and a function $a$ such that spt$\,a \subset V$ and $\int a \,d\mu_m = 0$. We need to show that
\begin{align*}
|\langle T(h_{I_1,\eta_1}& \otimes a), h_{I_2,\eta_2} \otimes \chi_{3V}\rangle| \\&+ |\langle T(h_{I_1,\eta_1} \otimes a), h_{I_2,\eta_2} \otimes \chi_{(3V)^c}\rangle| \lesssim A_{I_1I_2}^{\textup{sep}}\|a\|_{L^2(\mu_m)} \mu_m(3V)^{1/2}.
\end{align*}
There holds with an arbitrary $z \in I_1$ that
\begin{align*}
&|\langle T(h_{I_1,\eta_1} \otimes a), h_{I_2,\eta_2} \otimes \chi_{3V}\rangle| \\
&= \Big| \int_{I_1} \int_{I_2} [K_{a, \chi_{3V}}(x_1,y_1) - K_{a, \chi_{3V}}(x_1,z)]h_{I_1,\eta_1}(y_1)h_{I_2,\eta_2}(x_1)\,d\mu_n(x_1)\,d\mu_n(y_1)\Big|,
\end{align*}
and this implies that
\begin{displaymath}
|\langle T(h_{I_1,\eta_1} \otimes a), h_{I_2,\eta_2} \otimes \chi_{3V}\rangle|  \lesssim A_{I_1I_2}^{\textup{sep}} C(a,\chi_{3V}) \lesssim A_{I_1I_2}^{\textup{sep}} \|a\|_{L^2(\mu_m)} \mu_m(3V)^{1/2}.
\end{displaymath}

An easy consequence of the H\"older estimates of the kernel $K$ is that
\begin{align*}
&|\langle T(h_{I_1,\eta_1} \otimes a), h_{I_2,\eta_2} \otimes \chi_{(3V)^c}\rangle| \\
&\lesssim A_{I_1I_2}^{\textup{sep}} \int_V \int_{(3V)^c} \frac{\ell(V)^{\beta}}{|x_2-c_V|^{\beta}\lambda_m(c_V, |x_2-c_V|)} |a(y_2)|\,d\mu_m(x_2)\,d\mu_m(y_2) \\
&\lesssim A_{I_1I_2}^{\textup{sep}} \|a\|_{L^2(\mu_m)} \mu_m(V)^{1/2} \cdot \ell(V)^{\beta} \int_{(3V)^c} \frac{|x_2-c_V|^{-\beta}}{\lambda_m(c_V, |x_2-c_V|)}\,d\mu_m(x_2) \\
&\lesssim A_{I_1I_2}^{\textup{sep}} \|a\|_{L^2(\mu_m)} \mu_m(V)^{1/2}.
\end{align*}
This completes the proof of the lemma.
\end{proof}

Thus
\begin{equation*}
\begin{split}
  \Big(\sum_{I_2}\Big[  &\mathop{\sum_{I_1:\, \ell(I_1) \le  \ell(I_2)}}_{d(I_1, I_2) > 2\ell(I_1)^{\gamma_n}\ell(I_2)^{1-\gamma_n}}
   \|b_{I_1I_2}^{\eta_1\eta_2}\|_{\textup{BMO}^2_3(\mu_m)} \|f_{I_1}^{\eta_1}\|_{L^2(\mu_m)}\Big]^2\Big)^{1/2}  \\
  &\lesssim  \Big( \sum_{I_2}\Big[  \sum_{I_1:\, \ell(I_1) \le  \ell(I_2)}
   A_{I_1 I_2}^{\textup{sep}} \|f_{I_1}^{\eta_1}\|_{L^2(\mu_m)}\Big]^2\Big)^{1/2}  \\
   &\lesssim \Big(\sum_{I_1}\|f_{I_1}^{\eta_1}\|_{L^2(\mu_m)}^2\Big)^{1/2} \leq \|f\|_{L^2(\mu)},
\end{split}
\end{equation*}
where we used Proposition~\ref{l2s} in the second-to-last step, and the orthonormality of the functions $h_{I_1}^{\eta_1}$, implicit in the notation $f_{I_1}^{\eta_1}=\langle f,h_{I_1}^{\eta_1}\rangle_1$, in the last step.

\subsection{Separated/adjacent}
Unlike in the full diagonal, no surgery is needed in this part of the summation. Simply by using the kernel
$K_{u_{J_1, \kappa_1}, u_{J_2, \kappa_2}}$ and the knowledge that $C(u_{J_1, \kappa_1}, u_{J_2, \kappa_2}) \lesssim \|u_{J_1, \kappa_1}\|_{L^2(\mu_m)} \|u_{J_2, \kappa_2}\|_{L^2(\mu_m)} = 1$, we see like before
that separation in $\R^n$  gives
\begin{align*}
|\langle T(h_{I_1, \eta_1} \otimes u_{J_1, \kappa_1}),  h_{I_2, \eta_2} \otimes u_{J_2, \kappa_2}\rangle| \lesssim A_{I_1I_2}^{\textup{sep}}.
\end{align*}

For a given $J_1$ there are only boundedly many $J_2$ for which $2^{-r}\ell(J_2) \le \ell(J_1) \le \ell(J_2)$ and $d(J_1,J_2) \le 2\ell(J_1)^{\gamma_m}\ell(J_2)^{1-\gamma_m}$.
Combining this fact with H\"older's inequality and Proposition \ref{l2s} readily yields that the corresponding summation is dominated by $\|f\|_{L^2(\mu)} \|g\|_{L^2(\mu)}$.

\section{Adjacent cubes}\label{sec:adj}

We proceed to the case where the pair $(I_1,I_2)$ consists of adjacent cubes. Since all separated cases were already handled, this leaves us with the two possibilities for the other pair $(J_1,J_2)$ of being either nested or adjacent.

\subsection{Surgery for adjacent cubes}
For the analysis of the adjacent pairs $(I_1,I_2)$ handled in this section, we will need to carry out a further splitting introduced by Nazarov, Treil and Volberg under the name \emph{surgery}. We introduce it here in a pleasently compact form.

Consider two adjacent cubes $I_1,I_2$ with $\ell(I_1)\leq\ell(I_2)$. Let $\theta\in(0,1)$. We perform surgery on $(I_1, I_2)$ with parameter $\theta > 0$.
Let $j(\theta)\in\Z$ be such that $2^{-21}\theta\leq 2^{j(\theta)}<2^{-20}\theta$. Let $\mathcal{D}_n^*$ be yet another random grid in $\R^n$, independent of all other grids considered. Let $G:=\{g\in\mathcal{D}_n^*:\ell(g)=2^{j(\theta)}\ell(I_1)\}$, and for $x\in \R^n$, let $G(x)$ be the unique cube in $G$ that contains $x$. We define
\begin{equation*}
  I_{1,\partial}:=\{x\in I_1:d(G(x),\partial I_2)<\theta\ell(I_2)/2\}\cup\{x\in I_1\cap I_2: d(x,\partial G(x))<\theta\ell(G(x))\}.
\end{equation*}
Thus points in $I_{1,\partial}$ belong to $I_1$, and are either close to the boundary of $I_2$, or to the boundary of the grid $G$. The set $I_{1,\partial}$ depends on the set $I_2$ as well. However, we have
\begin{equation}\label{eq:Ibad}
\begin{split}
  I_{1,\partial}\subset I_{1,\textup{bad}}
  :=I_1\cap\Big[ &\bigcup_{\substack{I_2'\in\mathcal{D}_n'\\ \ell(I_1)\leq\ell(I_2')\leq 2^r\ell(I_1)}}\{x:d(x,\partial I_2')<\theta\ell(I_2')\} \\
    &\cup\bigcup_{\substack{g\in\mathcal{D}_n^*\\ \ell(g)=2^{j(\theta)}\ell(I_1)}}\{x:d(x,\partial g)<\theta\ell(g)\}\Big],
\end{split}
\end{equation}
which depends only on $I_1$ and the grids $\mathcal{D}_n'$ and $\mathcal{D}_n^*$.

We set
\begin{equation*}
  I_{1,\textup{sep}}:=I_1\setminus (I_{1,\partial}\cup I_2),
\end{equation*}
the part of $I_1$ strictly separated from $I_2$. Finally, we have
\begin{equation*}
  I_{1,\Delta} := I_1\setminus(I_{1,\partial}\cup I_{1,\textup{sep}})
  =\bigcup_i L_i,
\end{equation*}
where each $L_i$ is of the form $L_i=(1-\theta)g\cap I_1\cap I_2$ for some $g\in G$, and $\#i \lesssim_{\theta} 1$.
In fact, $L_i$ is of the form $L_i=(1-\theta)g$ unless it is close to the boundary of $I_1$; it cannot be close to the boundary of $I_2$, since such cubes were already subtracted in the $I_{1,\partial}$ component.

We have the partition
\begin{equation*}
  I_1=I_{1,\textup{sep}}\cup I_{1,\partial}\cup I_{1,\Delta} = I_{1,\textup{sep}}\cup I_{1,\partial}\cup \bigcup_i L_i ,
\end{equation*}
and in a completely analogous manner also
\begin{equation*}
  I_2=I_{2,\textup{sep}}\cup I_{2,\partial}\cup I_{2, \Delta} =  I_{2,\textup{sep}}\cup I_{2,\partial} \cup  \bigcup_j L_j.
\end{equation*}
A key observation is that all $L_i\subset I_1\cap I_2$ appearing in the first union are cubes (of the form $(1-\theta)g$ for $g\in G$) unless they are close to $\partial I_1$, and they are never close to $\partial I_2$, while the $L_j$ in the second union are cubes unless they are close to $\partial I_2$, and they are never close to $\partial I_1$. Thus, all $L_i=L_j$ that appear in both unions are cubes and then $5L_i \subset I_1 \cap I_2$.


\subsection{Adjacent/nested}
The structure of the decomposition is the same as in the separated/nested case.
Namely, we write $\langle T(h_{I_1, \eta_1} \otimes u_{J_1, \kappa_1}),  h_{I_2, \eta_2} \otimes u_{J_2, \kappa_2}\rangle$ as the sum
\begin{align*}
[\langle T(h_{I_1, \eta_1} \otimes u_{J_1, \kappa_1}),  h_{I_2, \eta_2} \otimes u_{J_2, \kappa_2}\rangle
- \langle &u_{J_2,\kappa_2} \rangle_{J_1} \langle T(h_{I_1, \eta_1} \otimes u_{J_1, \kappa_1}),  h_{I_2, \eta_2} \otimes 1\rangle] \\
&+ \langle u_{J_2,\kappa_2} \rangle_{J_1} \langle T(h_{I_1, \eta_1} \otimes u_{J_1, \kappa_1}),  h_{I_2, \eta_2} \otimes 1\rangle.
\end{align*}
Again, we start by dealing with the sum having the first term as the matrix element. Let $J_{2,1} \in \textup{ch}(J_2)$ be such that $J_1 \subset J_{2,1}$.
Note that $\langle u_{J_2,\kappa_2} \rangle_{J_1} = \langle u_{J_2,\kappa_2} \rangle_{J_{2,1}}$. Therefore, we have that
\begin{displaymath}
\langle T(h_{I_1, \eta_1} \otimes u_{J_1, \kappa_1}),  h_{I_2, \eta_2} \otimes u_{J_2, \kappa_2}\rangle
- \langle u_{J_2,\kappa_2} \rangle_{J_1} \langle T(h_{I_1, \eta_1} \otimes u_{J_1, \kappa_1}),  h_{I_2, \eta_2} \otimes 1\rangle
\end{displaymath}
equals
\begin{align*}
-\langle u_{J_2,\kappa_2} \rangle_{J_{2,1}}  \langle T(h_{I_1, \eta_1}& \otimes u_{J_1, \kappa_1}),  h_{I_2, \eta_2} \otimes \chi_{J_{2,1}^c}\rangle \\
&+ \mathop{\sum_{J_2' \in \textup{ch}(J_2)}}_{J_2' \subset J_2 \setminus J_{2,1}} \langle T(h_{I_1, \eta_1} \otimes u_{J_1, \kappa_1}),  h_{I_2, \eta_2} \otimes u_{J_2, \kappa_2}\chi_{J_2'}\rangle.
\end{align*}
As $J_1$ is good, $J_1 \subset J_{2,1}$ and $\ell(J_1) \le 2^{-r}\ell(J_{2,1})$, we have
\begin{displaymath}
d(J_1, J_{2,1}^c) \ge 4\ell(J_1)^{\gamma_m}\ell(J_{2,1})^{1-\gamma_m} > 2\ell(J_1)^{\gamma_m}\ell(J_2)^{1-\gamma_m}.
\end{displaymath}
We state the following analogs of Lemma \ref{mixedsep1} and Lemma \ref{mixedsep2}:
\begin{lem}
There holds
\begin{displaymath}
|\langle u_{J_2,\kappa_2} \rangle_{J_{2,1}}  \langle T(h_{I_1, \eta_1} \otimes u_{J_1, \kappa_1}),  h_{I_2, \eta_2} \otimes \chi_{J_{2,1}^c}\rangle| \lesssim A_{J_1J_2}^{\textup{in}}.
\end{displaymath}
\end{lem}

\begin{lem}
If $J_2' \in \textup{ch}(J_2)$ and $J_2' \subset J_2 \setminus J_{2,1}$, then
\begin{displaymath}
|\langle T(h_{I_1, \eta_1} \otimes u_{J_1, \kappa_1}),  h_{I_2, \eta_2} \otimes u_{J_2, \kappa_2}\chi_{J_2'}\rangle| \lesssim A_{J_1J_2}^{\textup{in}}.
\end{displaymath}
\end{lem}

Indeed, the proofs are similar as before except one uses the kernel $K_{h_{I_1, \eta_1},  h_{I_2, \eta_2}}$ and the fact that $C(h_{I_1, \eta_1},  h_{I_2, \eta_2}) \lesssim 1$.
Recalling Lemma \ref{InSummation} and the
fact that for a given $I_1$ there are only boundedly many $I_2$ for which $2^{-r}\ell(I_2) \le \ell(I_1) \le \ell(I_2)$ and $d(I_1,I_2) \le 2\ell(I_1)^{\gamma_n}\ell(I_2)^{1-\gamma_n}$, it is
immediate that the corresponding series is dominated by $\|f\|_{L^2(\mu)} \|g\|_{L^2(\mu)}$.

Therefore, we need only to consider the same summation but with the matrix element
\begin{displaymath}
\langle u_{J_2,\kappa_2} \rangle_{J_1} \langle T(h_{I_1, \eta_1} \otimes u_{J_1, \kappa_1}),  h_{I_2, \eta_2} \otimes 1\rangle.
\end{displaymath}
Exactly as we have seen before, the series with this matrix element collapses to
\begin{equation*}
\begin{split}
  &\Big\langle \sum_{\eta_1, \kappa_1, \eta_2} \mathop{\mathop{\sum_{I_1, I_2 \textup{ good}}}_{2^{-r}\ell(I_2) \le \ell(I_1) \le  \ell(I_2)}}_{d(I_1, I_2) \le 2\ell(I_1)^{\gamma_n}\ell(I_2)^{1-\gamma_n}}
h_{I_2, \eta_2} \otimes (\Pi^{\kappa_1}_{b_{I_1I_2}^{\eta_1\eta_2}})^*f_{I_1}^{\eta_1}, g_{\textup{good}}\Big\rangle \\
  &=\sum_{\eta_1,\kappa_1,\eta_2}
  \sum_{\substack{I_1,I_2\textup{ good}\\ I_1\sim I_2}}\langle (\Pi^{\kappa_1}_{b^{\eta_1\eta_2}_{I_1I_2}})^* f_{I_1}^{\eta_1}, g^{\eta_2}_{\textup{good},I_2}\rangle_2,
\end{split}
\end{equation*}
where $f_{I_1}^{\eta_1} = \langle f, h_{I_1,\eta_1} \rangle_1$, $g^{\eta_2}_{\textup{good},I_2}=\langle g_{\textup{good}}, h_{I_2, \eta_2}\rangle_1$,  and $b_{I_1I_2}^{\eta_1\eta_2} = \langle T^*(h_{I_2,\eta_2} \otimes 1), h_{I_1,\eta_1}\rangle_1$. The summing condition
tying $I_1$ and $I_2$ in the first line has been abbreviated to $I_1 \sim I_2$ in the second.

We can then estimate
\begin{equation*}
\begin{split}
  &\Big|\sum_{\substack{I_1,I_2\textup{ good}\\ I_1\sim I_2}}
    \langle (\Pi^{\kappa_1}_{b^{\eta_1\eta_2}_{I_1I_2}})^* f_{I_1}^{\eta_1}, g^{\eta_2}_{\textup{good},I_2}\rangle_2\Big| \\
  &\lesssim\sum_{\substack{I_1,I_2\\ I_1\sim I_2}}
    \| b^{\eta_1\eta_2}_{I_1I_2}\|_{\textup{BMO}^2_{15}(\mu_m)} \|f_{I_1}^{\eta_1}\|_{L^2(\mu_m)} \|g^{\eta_2}_{I_2}\|_{L^2(\mu_m)},
\end{split}
\end{equation*}
where we also used the simple bound that $\|g^{\eta_2}_{\textup{good},I_2}\|_{L^2(\mu_m)}\leq\|g^{\eta_2}_{I_2}\|_{L^2(\mu_m)}$, which follows at once from the Haar expansions.


\begin{lem}
There holds that
\begin{displaymath}
\|b_{I_1I_2}^{\eta_1\eta_2}\|_{\textup{BMO}^2_{15}(\mu_m)} \le C(\theta) + C\|T\| (H_{I_1, \textup{bad}}^{\eta_1}+H^{\eta_2}_{I_2,\textup{bad}})
\end{displaymath}
where
\begin{displaymath}
H_{I_i, \textup{bad}}^{\eta_i} = \Big( \sum_{I_i' \in \textup{ch}(I_i)} |\langle h_{I_i,\eta_i} \rangle_{I_i'} |^2 \mu_n(I_{i,\textup{bad}}') \Big)^{1/2}
\end{displaymath}
and $I_{i,\textup{bad}}$ is defined as in \eqref{eq:Ibad} (that is, we use surgery with parameter $\theta > 0$).
\end{lem}

\begin{proof}
We fix a cube $V \subset \R^m$ and a function $a$ such that spt$\,a \subset V$ and $\int a \,d\mu_m = 0$. We should show that
\begin{equation*}
  |\langle T(h_{I_1,\eta_1}\otimes a),h_{I_2,\eta_2}\otimes 1\rangle|
  \le  [C(\theta) + C\|T\| (H_{I_1, \textup{bad}}^{\eta_1}+H^{\eta_2}_{I_2,\textup{bad}})] \|a\|_{L^2(\mu_m)}\mu_m(15V)^{1/2}.
\end{equation*}
We split $1=\chi_{3V}+\chi_{(3V)^c}$. The H\"older property of $K_{h_{I_1, \eta_1},  h_{I_2, \eta_2}}$ gives that
\begin{displaymath}
|\langle T(h_{I_1,\eta_1} \otimes a), h_{I_2,\eta_2} \otimes \chi_{(3V)^c}\rangle| \lesssim \|a\|_{L^2(\mu_m)}\mu_m(V)^{1/2}.
\end{displaymath}

For the remaining part with $\chi_{3V}$ in place of $1$, writing the Haar functions $h_{I_i,\eta_i}$ as linear combinations of $\chi_{I_i'}$, where $I_i' \in \textup{ch}(I_i)$, we are reduced to bounding
\begin{equation*}
 \sum_{I_1 \in \textup{ch}(I_1)} \sum_{I_2 \in \textup{ch}(I_2)} |\ave{h_{I_1,\eta_1}}_{I_1'}| |\ave{h_{I_2,\eta_2}}_{I_2'}| |\langle T(\chi_{I_1'}\otimes a),\chi_{I_2'}\otimes \chi_{3V}\rangle|.
\end{equation*}
We concentrate on one such pairing
\begin{equation*}
  \langle T(\chi_{I_1}\otimes a),\chi_{I_2}\otimes \chi_{3V}\rangle,
\end{equation*}
where we have dropped the primes from $I_i'$ for brevity. This will be handled by applying the surgery to the pair $(I_1, I_2)$.

We write
\begin{align*}
\langle T(\chi_{I_1}\otimes a),\chi_{I_2}\otimes \chi_{3V}\rangle &= \sum_{\alpha \in \{\textup{sep},\,\partial\}} \langle T(\chi_{I_{1,\alpha}}\otimes a),\chi_{I_2}\otimes \chi_{3V}\rangle \\
&+ \sum_{\beta \in \{\textup{sep},\,\partial\}} \langle T(\chi_{I_{1,\Delta}}\otimes a),\chi_{I_{2, \beta}}\otimes \chi_{3V}\rangle \\
&+ \sum_{i \ne j} \langle T(\chi_{L_i}\otimes a),\chi_{L_j}\otimes \chi_{3V}\rangle\\
& + \sum_{i = j} \langle T(\chi_{L_i}\otimes a),\chi_{L_i}\otimes \chi_{3V}\rangle.
\end{align*}
If $\alpha = \textup{sep}$ or $\beta = \textup{sep}$ or $i \ne j$, then the corresponding pairing is seen to be dominated by
\begin{displaymath}
C(\theta)\mu_n(I_1)^{1/2}\mu_n(I_2)^{1/2}\|a\|_{L^2(\mu_m)}\mu_m(3V)^{1/2}
\end{displaymath}
using the size estimate of the kernel $K_{a, \chi_{3V}}$ together with the fact that the sets are separated by $c(\theta)\ell(I_1) \sim c(\theta)\ell(I_2)$.
In the case $i \ne j$ a further large dependence on $\theta$ is gained from the summation $\sum_{i \ne j} 1$.

If $i = j$, then by the diagonal BMO assumptions the corresponding pairing is dominated by
\begin{displaymath}
C\mu_n(5L_i)\|a\|_{L^2(\mu_m)}\mu_m(15V)^{1/2} \le C\mu_n(I_1)^{1/2}\mu_n(I_2)^{1/2}\|a\|_{L^2(\mu_m)}\mu_m(15V)^{1/2}.
\end{displaymath}
A factor $C(\theta)$ is gained from the summation $\sum_{i = j} 1$.

Finally, the sum of the cases $\alpha = \partial$ and $\beta = \partial$ is dominated by
\begin{displaymath}
\|T\|( \mu_n(I_{1,\textup{bad}})^{1/2} \mu_n(I_2)^{1/2} +  \mu_n(I_1)^{1/2}\mu_n(I_{2,\textup{bad}})^{1/2}) \|a\|_{L^2(\mu_m)}\mu_m(3V)^{1/2}.
\end{displaymath}
The BMO estimate now readily follows.

\end{proof}

Recalling that
\begin{equation*}
  \sum_{I_1 \in \mathcal{D}_{n}} \|f_{I_1}^{\eta_1}\|_{L^2(\mu_m)}^2 \leq \|f\|_{L^2(\mu)}^2,
  \qquad\sum_{I_2 \in \mathcal{D}_{n}'} \|g_{I_2}^{\eta_2}\|_{L^2(\mu_m)}^2 \leq \|g\|_{L^2(\mu)}^2,
\end{equation*}
we now know that
\begin{align*}
&\sum_{I_2 \in \mathcal{D}_{n}'} \mathop{\sum_{I_1 \in \mathcal{D}_{n}}}_{I_1 \sim I_2} \| b^{\eta_1\eta_2}_{I_1I_2}\|_{\textup{BMO}^2_{15}(\mu_m)}
  \|f_{I_1}^{\eta_1}\|_{L^2(\mu_m)}  \|g_{I_2}^{\eta_2}\|_{L^2(\mu_m)}  \\
&\le\sum_{I_2 \in \mathcal{D}_{n}'} \mathop{\sum_{I_1 \in \mathcal{D}_{n}}}_{I_1 \sim I_2}\big(C(\theta)+C\|T\|(H^{\eta_1}_{I_{1,\textup{bad}}}+
    H^{\eta_2}_{I_{2,\textup{bad}}})\big)  \|f_{I_1}^{\eta_1}\|_{L^2(\mu_m)}  \|g_{I_2}^{\eta_2}\|_{L^2(\mu_m)}  \\
&\le C(\theta)\|f\|_{L^2(\mu)}\|g\|_{L^2(\mu)}
  +C\|T\|\Big[ \sum_{I_1 \in \mathcal{D}_{n}} (H_{I_1, \textup{bad}}^{\eta_1})^2 \|f_{I_1}^{\eta_1}\|_{L^2(\mu_m)}^2 \bigg]^{1/2}\|g\|_{L^2(\mu)} \\
 &\qquad +C\|T\|\|f\|_{L^2(\mu)}\Big[ \sum_{I_2 \in \mathcal{D}_{n}'} (H_{I_2, \textup{bad}}^{\eta_2})^2 \|g_{I_2}^{\eta_2}\|_{L^2(\mu_m)}^2 \bigg]^{1/2}.
\end{align*}


The last two terms are symmetric to each other. For example, the first one can be estimated by
\begin{align*}
E_{w_n'}E_{w_n^*}[(H_{I_1, \textup{bad}}^{\eta_1})^2]
&= \sum_{I_1' \in \textup{ch}(I_1)} |\langle h_{I_1,\eta_1} \rangle_{I_1'} |^2 E_{w_n'}E_{w_n^*}[\mu_n(I_{1,\textup{bad}}')] \\
&\le c(\theta) \sum_{I_1' \in \textup{ch}(I_1)} |\langle h_{I_1,\eta_1} \rangle_{I_1'} |^2 \mu_n(I_1') \le c(\theta),
\end{align*}
where $c(\theta) \to 0$, when $\theta \to 0$. Furthermore, we have
\begin{equation*}
\begin{split}
   E_{w_n'}E_{w_n^*}&\Big[ \sum_{I_1 \in \mathcal{D}_{n}} (H_{I_1, \textup{bad}}^{\eta_1})^2 \|f_{I_1}^{\eta_1}\|_{L^2(\mu_m)}^2 \bigg]^{1/2} \\
   &\leq   \Big[ \sum_{I_1 \in \mathcal{D}_{n}} E_{w_n'}E_{w_n^*}[(H_{I_1, \textup{bad}}^{\eta_1})^2] \|f_{I_1}^{\eta_1}\|_{L^2(\mu_m)}^2 \bigg]^{1/2} \\
   &\leq   c(\theta)\Big[ \sum_{I_1 \in \mathcal{D}_{n}} \|f_{I_1}^{\eta_1}\|_{L^2(\mu_m)}^2 \bigg]^{1/2}
   \leq c(\theta)\|f\|_{L^2(\mu)}.
\end{split}
\end{equation*}

We have proved that the expectation of the series with the matrix element
\begin{displaymath}
\langle u_{J_2,\kappa_2} \rangle_{J_1} \langle T(h_{I_1, \eta_1} \otimes u_{J_1, \kappa_1}),  h_{I_2, \eta_2} \otimes 1\rangle
\end{displaymath}
is dominated by $C(\theta) \|f\|_{L^2(\mu)} \|g\|_{L^2(\mu)} + c(\theta) \|T\| \|f\|_{L^2(\mu)} \|g\|_{L^2(\mu)}$. Fixing $\theta$ small we establish the bound
$C\|f\|_{L^2(\mu)} \|g\|_{L^2(\mu)} + 2^{-100}\|T\| \|f\|_{L^2(\mu)} \|g\|_{L^2(\mu)}$.

\subsection{Adjacent/adjacent}

We need to estimate the pairing
\begin{equation*}
\begin{split}
  &\langle T(h_{I_1,\eta_1}\otimes u_{J_1,\kappa_1}) , h_{I_2,\eta_2}\otimes u_{J_2,\kappa_2} \rangle \\
  &=\sum_{I_1',I_2',J_1',J_2'}
      \ave{h_{I_1,\eta_1}}_{I_1'}\ave{u_{J_1,\kappa_1}}_{J_1'}
      \ave{h_{I_2,\eta_2}}_{I_2'}\ave{u_{J_2,\kappa_2}}_{J_2'}
      \ave{ T(\chi_{I_1'}\otimes \chi_{J_1'}), \chi_{I_2'}\otimes \chi_{J_2'} },
\end{split}
\end{equation*}
where the summation is over all dyadic children $I_1'$ of $I_1$, $I_2'$ of $I_2$ etc.

\begin{lem}
For adjacent pairs $(I_1,I_2)$ and $(J_1,J_2)$, we have
\begin{equation*}
\begin{split}
  &\abs{\langle T(h_{I_1,\eta_1}\otimes u_{J_1,\kappa_1}) , h_{I_2,\eta_2}\otimes u_{J_2,\kappa_2} \rangle} \\
  &\le C(\epsilon)+C\|T\|\big(H^{\eta_1}_{I_1,\textup{bad}}+H^{\eta_2}_{I_2,\textup{bad}}+H^{\kappa_1}_{J_1,\textup{bad}}+H^{\kappa_2}_{J_2,\textup{bad}}\big),
\end{split}
\end{equation*}
where $\epsilon > 0$ is the surgery parameter both on $\R^n$ and $\R^m$.
\end{lem}
\begin{proof}
In the surgery of $(J_1, J_2)$ we write $K_i$ for the sets that correspond to the sets $L_i$ in the surgery of $(I_1, I_2)$. We now decompose
\begin{align*}
\langle T(\chi_{I_1}\otimes \chi_{J_1}),\chi_{I_2}\otimes \chi_{J_2}\rangle &= \sum_{\alpha_1 \in \{\textup{sep},\,\partial\}} \langle T(\chi_{I_{1,\alpha_1}}\otimes \chi_{J_1}),\chi_{I_2}\otimes \chi_{J_2}\rangle \\
&+ \sum_{\beta_1 \in \{\textup{sep},\,\partial\}} \langle T(\chi_{I_{1,\Delta}}\otimes \chi_{J_1}),\chi_{I_{2, \beta_1}}\otimes \chi_{J_2}\rangle \\
&+ \sum_{\alpha_2 \in \{\textup{sep},\,\partial\}}  \langle T(\chi_{I_{1,\Delta}}\otimes \chi_{J_{1, \alpha_2}}),\chi_{I_{2,\Delta}}\otimes \chi_{J_2}\rangle \\
&+ \sum_{\beta_2 \in \{\textup{sep},\,\partial\}} \langle T(\chi_{I_{1, \Delta}}\otimes \chi_{J_{1, \Delta}}),\chi_{I_{2,\Delta}}\otimes \chi_{J_{2, \beta_2}}\rangle \\
&+ \sum_{i_1 \ne j_1}  \langle T(\chi_{L_{i_1}}\otimes \chi_{J_{1, \Delta}}),\chi_{L_{j_1}}\otimes \chi_{J_{2, \Delta}}\rangle \\
&+ \sum_{i_1 = j_1} \sum_{i_2 \ne j_2} \langle T(\chi_{L_{i_1}}\otimes \chi_{K_{i_2}}),\chi_{L_{i_1}}\otimes \chi_{K_{j_2}}\rangle \\
&+ \sum_{i_1 = j_1} \sum_{i_2 = j_2}  \langle T(\chi_{L_{i_1}}\otimes \chi_{K_{i_2}}),\chi_{L_{i_1}}\otimes \chi_{K_{i_2}}\rangle.
\end{align*}

If $i_1 = j_1$ and $i_2 = j_2$, then the weak boundedness property gives that the corresponding pairing is dominated by
\begin{displaymath}
C\mu_n(5L_{i_1}) \mu_m(5K_{i_2}) \le C\mu_n(I_1)^{1/2}\mu_n(I_2)^{1/2}\mu_m(J_1)^{1/2}\mu_m(J_2)^{1/2}.
\end{displaymath}
A factor $C(\epsilon)$ is picked up from the summation $\sum_{i_1 = j_1} \sum_{i_2 = j_2} 1$.

The sum of the cases $\alpha_1 = \partial$ and $\beta_1 = \partial$ is dominated by
\begin{displaymath}
\|T\|( \mu_n(I_{1,\textup{bad}})^{1/2} \mu_n(I_2)^{1/2} +  \mu_n(I_1)^{1/2}\mu_n(I_{2,\textup{bad}})^{1/2}) \mu_m(J_1)^{1/2}\mu_m(J_2)^{1/2}.
\end{displaymath}
The sum of the cases $\alpha_2 = \partial$ and $\beta_2 = \partial$ is dominated by
\begin{displaymath}
\|T\| \mu_n(I_1)^{1/2}\mu_n(I_2)^{1/2} ( \mu_m(J_{1,\textup{bad}})^{1/2} \mu_m(J_2)^{1/2} +  \mu_m(J_1)^{1/2}\mu_m(J_{2,\textup{bad}})^{1/2}).
\end{displaymath}

In all the other cases we have separation in $\R^n$ by $\epsilon\ell(I_1) \sim \epsilon \ell(I_2)$ or separation in $\R^m$ by $\epsilon\ell(J_1) \sim \epsilon\ell(J_2)$.
Partial kernel representations give the bound
\begin{displaymath}
C(\epsilon)\mu_n(I_1)^{1/2}\mu_n(I_2)^{1/2}\mu_m(J_1)^{1/2}\mu_m(J_2)^{1/2}.
\end{displaymath}

We have shown that
\begin{align*}
|\langle &T(\chi_{I_1}\otimes \chi_{J_1}),\chi_{I_2}\otimes \chi_{J_2}\rangle| \\
&\le C(\epsilon)\mu_n(I_1)^{1/2}\mu_n(I_2)^{1/2}\mu_m(J_1)^{1/2}\mu_m(J_2)^{1/2} \\
&+ \|T\| ( \mu_n(I_{1,\textup{bad}})^{1/2} \mu_n(I_2)^{1/2} +  \mu_n(I_1)^{1/2}\mu_n(I_{2,\textup{bad}})^{1/2}) \mu_m(J_1)^{1/2}\mu_m(J_2)^{1/2} \\
&+ \|T\| \mu_n(I_1)^{1/2}\mu_n(I_2)^{1/2} ( \mu_m(J_{1,\textup{bad}})^{1/2} \mu_m(J_2)^{1/2} +  \mu_m(J_1)^{1/2}\mu_m(J_{2,\textup{bad}})^{1/2}).
\end{align*}
Apply this to $|\ave{ T(\chi_{I_1'}\otimes \chi_{J_1'}), \chi_{I_2'}\otimes \chi_{J_2'} }|$, where $I_1' \in \textup{ch}(I_1)$, $I_2' \in \textup{ch}(I_2)$, $J_1' \in \textup{ch}(J_1)$ and $J_2' \in \textup{ch}(J_2)$
to get the claim.
\end{proof}

With this lemma at hand, we see in the same way as in the adjacent/nested case that
\begin{align*}
E\Big|  \sum_{I_1,I_2:I_1\sim I_2}\sum_{J_1,J_2:J_1\sim J_2}
   &\langle f, h_{I_1,\eta_1}\otimes u_{J_1,\kappa_1}\rangle \langle g, h_{I_2,\eta_2}\otimes u_{J_2,\kappa_2}\rangle \\
   &\times\langle T(h_{I_1,\eta_1}\otimes u_{J_1,\kappa_1}) , h_{I_2,\eta_2}\otimes u_{J_2,\kappa_2}\rangle\Big|
\end{align*}
is dominated by $C(\epsilon) \|f\|_{L^2(\mu)} \|g\|_{L^2(\mu)} + c(\epsilon) \|T\| \|f\|_{L^2(\mu)} \|g\|_{L^2(\mu)}$, where $c(\epsilon) \to 0$ when $\epsilon \to 0$.
Again, fixing $\epsilon$ small we establish the bound $C\|f\|_{L^2(\mu)} \|g\|_{L^2(\mu)} + 2^{-100}\|T\| \|f\|_{L^2(\mu)} \|g\|_{L^2(\mu)}$.

\section{Nested cubes}\label{sec:nest}

The only case left to deal with consists of quadruples of cubes such that $I_1\subset I_{2,1} \in \textup{ch}(I_2)$ and $J_1\subset J_{2,1} \in \textup{ch}(J_2)$.
We then split
\begin{align*}
\langle T(h_{I_1, \eta_1} \otimes u_{J_1, \kappa_1}),  h_{I_2, \eta_2} \otimes u_{J_2, \kappa_2}\rangle &= \langle T(h_{I_1, \eta_1} \otimes u_{J_1, \kappa_1}),  s_{I_1I_2}^{\eta_2} \otimes s_{J_1J_2}^{\kappa_2} \rangle \\
&+ \langle h_{I_2, \eta_2} \rangle_{I_1} \langle T(h_{I_1, \eta_1} \otimes u_{J_1, \kappa_1}),  1 \otimes s_{J_1J_2}^{\kappa_2} \rangle \\
&+ \langle u_{J_2, \kappa_2} \rangle_{J_1} \langle T(h_{I_1, \eta_1} \otimes u_{J_1, \kappa_1}),  s_{I_1I_2}^{\eta_2} \otimes 1 \rangle \\
&+ \langle h_{I_2, \eta_2} \rangle_{I_1} \langle u_{J_2, \kappa_2} \rangle_{J_1} \langle T(h_{I_1, \eta_1} \otimes u_{J_1, \kappa_1}),  1 \rangle,
\end{align*}
where
\begin{displaymath}
s_{I_1I_2}^{\eta_2} = \chi_{I_{2,1}^c}(h_{I_2, \eta_2} - \langle h_{I_2, \eta_2} \rangle_{I_{2,1}}) = -\langle h_{I_2, \eta_2} \rangle_{I_{2,1}} \chi_{I_{2,1}^c}
 + \mathop{\sum_{I_2' \in \textup{ch}(I_2)}}_{I_2' \subset I_2 \setminus I_{2,1}} h_{I_2, \eta_2} \chi_{I_2'}
\end{displaymath}
and
\begin{displaymath}
s_{J_1J_2}^{\kappa_2} = \chi_{J_{2,1}^c}(u_{J_2, \kappa_2} - \langle u_{J_2, \kappa_2} \rangle_{J_{2,1}}) = -\langle u_{J_2, \kappa_2} \rangle_{J_{2,1}} \chi_{J_{2,1}^c}
 + \mathop{\sum_{J_2' \in \textup{ch}(J_2)}}_{J_2' \subset J_2 \setminus J_{2,1}} u_{J_2, \kappa_2} \chi_{J_2'}.
\end{displaymath}

\subsection{Terms involving separation}
We deal with the first three terms in the splitting, where at least one of $s^{\eta_2}_{I_1I_2}$ or $s^{\kappa_2}_{J_1J_2}$ appears.

We first consider the sum having the matrix element $ \langle T(h_{I_1, \eta_1} \otimes u_{J_1, \kappa_1}),  s_{I_1I_2}^{\eta_2} \otimes s_{J_1J_2}^{\kappa_2} \rangle$.
\begin{lem}
There holds that
\begin{displaymath}
|\langle T(h_{I_1, \eta_1} \otimes u_{J_1, \kappa_1}),  s_{I_1I_2}^{\eta_2} \otimes s_{J_1J_2}^{\kappa_2} \rangle| \lesssim A_{I_1I_2}^{\textup{in}} A_{J_1J_2}^{\textup{in}}.
\end{displaymath}
\end{lem}
\begin{proof}
We write out $\langle T(h_{I_1, \eta_1} \otimes u_{J_1, \kappa_1}),  s_{I_1I_2}^{\eta_2} \otimes s_{J_1J_2}^{\kappa_2} \rangle$ as the following sum of four terms:
\begin{align*}
& \langle h_{I_2, \eta_2} \rangle_{I_{2,1}} \langle u_{J_2, \kappa_2} \rangle_{J_{2,1}} \langle T(h_{I_1, \eta_1} \otimes u_{J_1, \kappa_1}),  \chi_{I_{2,1}^c} \otimes  \chi_{J_{2,1}^c}\rangle \\
&- \langle h_{I_2, \eta_2} \rangle_{I_{2,1}} \mathop{\sum_{J_2' \in \textup{ch}(J_2)}}_{J_2' \subset J_2 \setminus J_{2,1}}  \langle T(h_{I_1, \eta_1} \otimes u_{J_1, \kappa_1}),  \chi_{I_{2,1}^c} \otimes  u_{J_2, \kappa_2} \chi_{J_2'}\rangle \\
&- \langle u_{J_2, \kappa_2} \rangle_{J_{2,1}} \mathop{\sum_{I_2' \in \textup{ch}(I_2)}}_{I_2' \subset I_2 \setminus I_{2,1}} \langle T(h_{I_1, \eta_1} \otimes u_{J_1, \kappa_1}),  h_{I_2, \eta_2} \chi_{I_2'} \otimes  \chi_{J_{2,1}^c}\rangle \\
&+  \mathop{\sum_{I_2' \in \textup{ch}(I_2)}}_{I_2' \subset I_2 \setminus I_{2,1}} \mathop{\sum_{J_2' \in \textup{ch}(J_2)}}_{J_2' \subset J_2 \setminus J_{2,1}} 
\langle T(h_{I_1, \eta_1} \otimes u_{J_1, \kappa_1}),  h_{I_2, \eta_2} \chi_{I_2'} \otimes  u_{J_2, \kappa_2} \chi_{J_2'}\rangle.
\end{align*}

We handle the second term here (the rest are bounded using the same tools). We record that
\begin{displaymath}
d(I_1, I_{2,1}^c) \ge 4\ell(I_1)^{\gamma_n}\ell(I_{2,1})^{1-\gamma_n} > 2\ell(I_1)^{\gamma_n}\ell(I_2)^{1-\gamma_n}
\end{displaymath}
and
\begin{align*}
d(J_1, J_2') \ge d(J_1, J_{2,1}^c) &\ge 4\ell(J_1)^{\gamma_m}\ell(J_{2,1})^{1-\gamma_m} \\ &= 4\ell(J_1)^{\gamma_m}\ell(J_2')^{1-\gamma_m}  > 2\ell(J_1)^{\gamma_m}\ell(J_2)^{1-\gamma_m}.
\end{align*}
Recall that $\eta_1 \ne 0$ and $\kappa_1 \ne 0$ (that is, the corresponding Haar functions have zero means), since for example $I_1$ is (much) smaller than $I_2$ and so $\ell(I_1) < 2^{\ell}$.
Thus, with any $z \in I_1$ and $w \in J_1$ we may write
\begin{align*}
&\langle T(h_{I_1, \eta_1} \otimes u_{J_1, \kappa_1}),  \chi_{I_{2,1}^c} \otimes  u_{J_2, \kappa_2} \chi_{J_2'}\rangle \\
&= \int_{I_1} \int_{J_1} \int_{I_{2,1}^c} \int_{J_2'} [K(x,y) - K(x, (y_1, w)) - K(x, (z, y_2)) + K(x, (z,w))] \\
&\hspace{3.4cm}\times h_{I_1, \eta_1}(y_1) u_{J_1, \kappa_1}(y_2)  u_{J_2, \kappa_2}(x_2) \,d\mu_m(x_2)\,d\mu_n(x_1)\,d\mu_m(y_2) \, d\mu_n(y_1).
\end{align*}
We get the bound
\begin{align*}
&|\langle T(h_{I_1, \eta_1} \otimes u_{J_1, \kappa_1}),  \chi_{I_{2,1}^c} \otimes  u_{J_2, \kappa_2} \chi_{J_2'}\rangle| \\
&\lesssim \mu_n(I_1)^{1/2}\ell(I_1)^{\alpha} \int_{\R^n \setminus B(z, d(I_1, I_{2,1}^c))} \frac{|x_1 -z|^{-\alpha}}{\lambda_n(z, |x_1-z|)}\,d\mu_n(x_1) \\
& \hspace{4cm} \times \mu_m(J_1)^{1/2}\mu_m(J_2')^{1/2}\frac{\ell(J_1)^{\beta}}{d(J_1,J_2')^{\beta}\lambda_m(w, d(J_1,J_2'))}.
\end{align*}
As in the proof of Lemma \ref{ccss} we get that
\begin{align*}
\frac{\ell(J_1)^{\beta}}{d(J_1,J_2')^{\beta}\lambda_m(w, d(J_1,J_2'))} &\lesssim \frac{\ell(J_1)^{\beta/2}\ell(J_2')^{\beta/2}}{D(J_1,J_2')^{\beta}\lambda_m(w, D(J_1,J_2'))} \\
&\le \frac{\ell(J_1)^{\beta/2}}{\ell(J_2')^{\beta/2}\lambda_m(w, \ell(J_2'))}.
\end{align*}
As we have also previously seen, there holds that
\begin{displaymath}
\int_{\R^n \setminus B(z, d(I_1, I_{2,1}^c))} \frac{|x_1 -z|^{-\alpha}}{\lambda_n(z, |x_1-z|)}\,d\mu_n(x_1) \lesssim d(I_1, I_{2,1}^c)^{-\alpha} \lesssim \ell(I_1)^{-\alpha/2}\ell(I_2)^{-\alpha/2}
\end{displaymath}
and
\begin{displaymath}
\frac{\mu_m(J_2')^{1/2}}{\lambda_m(w, \ell(J_2'))} \lesssim \mu_m(J_{2,1})^{-1/2}.
\end{displaymath}
It remains to recall that $|\langle h_{I_2,\eta_2} \rangle_{I_{2,1}}| \le \mu_n(I_{2,1})^{-1/2}$. Indeed, putting these estimates together readily yields the bound $A_{I_1I_2}^{\textup{in}} A_{J_1J_2}^{\textup{in}}$.
\end{proof}
Combining the previous lemma with Lemma \ref{InSummation} immediately gives that the summation with the matrix element $\langle T(h_{I_1, \eta_1} \otimes u_{J_1, \kappa_1}),  s_{I_1I_2}^{\eta_2} \otimes s_{J_1J_2}^{\kappa_2} \rangle$
is dominated by $\|f\|_{L^2(\mu)}\|g\|_{L^2(\mu)}$.

We now deal with the sum having the matrix element $\langle u_{J_2, \kappa_2} \rangle_{J_1} \langle T(h_{I_1, \eta_1} \otimes u_{J_1, \kappa_1}),  s_{I_1I_2}^{\eta_2} \otimes 1 \rangle$.
Just like in the separated/nested case, we find out that the corresponding sum collapses to
\begin{displaymath}
\Big\langle \sum_{\eta_1, \kappa_1, \eta_2} \mathop{\mathop{\sum_{I_1, I_2 \textup{ good}}}_{\ell(I_1) <  2^{-r}\ell(I_2)}}_{I_1 \subset I_2}
h_{I_2, \eta_2} \otimes (\Pi^{\kappa_1}_{b_{I_1I_2}^{\eta_1\eta_2}})^*f_{I_1}^{\eta_1}, g_{\textup{good}}\Big\rangle,
\end{displaymath}
where $f_{I_1}^{\eta_1} = \langle f, h_{I_1,\eta_1} \rangle_1$ and $b_{I_1I_2}^{\eta_1\eta_2} = \langle T^*(s_{I_1I_2}^{\eta_2} \otimes 1), h_{I_1,\eta_1}\rangle_1$.
Thus, it is enough to fix $\eta_1 \ne 0$, $\kappa_1 \ne 0$ and $\eta_2$, and then show that
\begin{displaymath}
\Big\|  \mathop{\mathop{\sum_{I_1, I_2 \textup{ good}}}_{\ell(I_1) <  2^{-r}\ell(I_2)}}_{I_1 \subset I_2}
h_{I_2, \eta_2} \otimes (\Pi^{\kappa_1}_{b_{I_1I_2}^{\eta_1\eta_2}})^*f_{I_1}^{\eta_1} \Big\|_{L^2(\mu)} \lesssim \|f\|_{L^2(\mu)}.
\end{displaymath}

An estimate of by now familiar nature shows that $\|b_{I_1I_2}^{\eta_1\eta_2}\|_{\textup{BMO}^2_3(\mu_m)} \lesssim A_{I_1I_2}^{\textup{in}}$.
We have by orthonormality and Lemma~\ref{InSummation} that
\begin{align*}
&\Big\|  \mathop{\mathop{\sum_{I_1, I_2 \textup{ good}}}_{\ell(I_1) <  2^{-r}\ell(I_2)}}_{I_1 \subset I_2}
h_{I_2, \eta_2} \otimes (\Pi^{\kappa_1}_{b_{I_1I_2}^{\eta_1\eta_2}})^*f_{I_1}^{\eta_1} \Big\|_{L^2(\mu)} \\
&\lesssim\Big(\sum_{I_2\textup{ good}}\Big[\sum_{\substack{I_1\textup{ good}\\ \ell(I_1)<2^{-r}\ell(I_2) \\ I_1\subset I_2 }} A_{I_1 I_2}^{\textup{in}}
   \|f^{\eta_1}_{I_1}\|_{L^2(\mu_m)}\Big]^2\Big)^{1/2} \\
 &\lesssim\Big(\sum_{I_1}\|f^{\eta_1}_{I_1}\|_{L^2(\mu_m)}^2\Big)^{1/2}\leq \|f\|_{L^2(\mu)}.
\end{align*}

Note that the proof of the boundedness of the summation with the matrix element $\langle h_{I_2, \eta_2} \rangle_{I_1} \langle T(h_{I_1, \eta_1} \otimes u_{J_1, \kappa_1}),  1 \otimes s_{J_1J_2}^{\kappa_2} \rangle$
is analogous.

\subsection{Full paraproducts}
We are reduced to considering the summation with the matrix element $\langle h_{I_2, \eta_2} \rangle_{I_1} \langle u_{J_2, \kappa_2} \rangle_{J_1} \langle T(h_{I_1, \eta_1} \otimes u_{J_1, \kappa_1}),  1 \rangle$.
We fix $\eta_1 \ne 0$ and $\kappa_1 \ne 0$. We need to simplify the summation
\begin{align*}
\sum_{\eta_2, \kappa_2} \sum_{I_1\in  \mathcal{D}_{n,\textup{good}}}& \sum_{J_1 \in  \mathcal{D}_{m,\textup{good}}} \mathop{\sum_{I_2 \in \mathcal{D}_n'}}_{2^r\ell(I_1) < \ell(I_2) \le 2^{\ell}} \mathop{\sum_{J_2 \in \mathcal{D}_m'}}_{2^r\ell(J_1) < \ell(J_2) \le 2^{\ell}}
\langle h_{I_2, \eta_2} \rangle_{I_1} \langle u_{J_2, \kappa_2} \rangle_{J_1} \\
& \times \langle g_{\textup{good}}, h_{I_2, \eta_2} \otimes u_{J_2, \kappa_2}\rangle \langle T^*1, h_{I_1, \eta_1} \otimes u_{J_1, \kappa_1} \rangle
\langle f, h_{I_1, \eta_1} \otimes u_{J_1, \kappa_1}\rangle.
\end{align*}
Notice that
\begin{displaymath}
\Big\langle \mathop{\sum_{J_2 \in \mathcal{D}_m'}}_{2^r\ell(J_1) < \ell(J_2) \le 2^{\ell}} \sum_{\kappa_2} \langle g_{\textup{good}}, h_{I_2, \eta_2} \otimes u_{J_2, \kappa_2}\rangle u_{J_2, \kappa_2} \Big\rangle_{J_1}
= \langle \langle g_{\textup{good}}, h_{I_2, \eta_2} \rangle_1 \rangle_{S(J_1)},
\end{displaymath}
where $S(J_1) \in \mathcal{D}_m'$ is the unique cube for which $\ell(S(J_1)) = 2^r\ell(J_1)$ and $J_1 \subset S(J_1)$. Continue to notice that
\begin{align*}
\Big\langle  \mathop{\sum_{I_2 \in \mathcal{D}_n'}}_{2^r\ell(I_1) < \ell(I_2) \le 2^{\ell}} \sum_{\eta_2}  h_{I_2, \eta_2} \otimes \langle g_{\textup{good}}, h_{I_2, \eta_2} \rangle_1 \Big\rangle_{I_1 \times S(J_1)} = \langle g_{\textup{good}} \rangle_{S(I_1) \times S(J_1)},
\end{align*}
where $S(I_1) \in \mathcal{D}_n'$ is the unique cube for which $\ell(S(I_1)) = 2^r\ell(I_1)$ and $I_1 \subset S(I_1)$. Therefore, the expression we started with collapses to
\begin{align*}
 &\sum_{I_1\in  \mathcal{D}_{n,\textup{good}}} \sum_{J_1 \in  \mathcal{D}_{m,\textup{good}}} \langle g_{\textup{good}} \rangle_{S(I_1) \times S(J_1)} \langle T^*1, h_{I_1, \eta_1} \otimes u_{J_1, \kappa_1} \rangle 
\langle f, h_{I_1, \eta_1} \otimes u_{J_1, \kappa_1}\rangle \\
&= \sum_{I_2 \in \mathcal{D}_n'}  \mathop{\mathop{\sum_{I_1 \in \mathcal{D}_{n, \textup{good}}}}_{I_1 \subset I_2}}_{\ell(I_1) = 2^{-r}\ell(I_2)} \sum_{J_2 \in \mathcal{D}_m'}  \mathop{\mathop{\sum_{J_1 \in \mathcal{D}_{m, \textup{good}}}}_{J_1 \subset J_2}}_{\ell(J_1) = 2^{-r}\ell(J_2)}
\langle g_{\textup{good}} \rangle_{I_2 \times J_2} \langle T^*1, h_{I_1, \eta_1} \otimes u_{J_1, \kappa_1} \rangle \\
& \hspace{9.5cm} \times \langle f, h_{I_1, \eta_1} \otimes u_{J_1, \kappa_1}\rangle.
\end{align*}
One can write this as a pairing $\langle f, \Pi_{T^*1}^{\eta_1\kappa_1} g_{\textup{good}}\rangle$, where
\begin{displaymath}
\Pi_b^{\eta_1\kappa_1} u = \sum_{I_2 \in \mathcal{D}_n'}  \mathop{\mathop{\sum_{I_1 \in \mathcal{D}_{n, \textup{good}}}}_{I_1 \subset I_2}}_{\ell(I_1) = 2^{-r}\ell(I_2)} \sum_{J_2 \in \mathcal{D}_m'}
\mathop{\mathop{\sum_{J_1 \in \mathcal{D}_{m, \textup{good}}}}_{J_1 \subset J_2}}_{\ell(J_1) = 2^{-r}\ell(J_2)} \langle u \rangle_{I_2 \times J_2} \langle b, h_{I_1, \eta_1} \otimes u_{J_1, \kappa_1} \rangle
h_{I_1, \eta_1} \otimes u_{J_1, \kappa_1}.
\end{displaymath}
\begin{prop}
There holds that
\begin{displaymath}
\|\Pi_b^{\eta_1\kappa_1} u\|_{L^2(\mu)} \le 4\|b\|_{\textup{BMO}_{\textup{prod}}(\mu)} \|u\|_{L^2(\mu)}, \qquad \eta_1, \kappa_1 \ne 0.
\end{displaymath}
\end{prop}
\begin{proof}
There holds that
\begin{align*}
\|\Pi_b^{\eta_1\kappa_1} u\|_{L^2(\mu)}^2 &\le \sum_{S \in \mathcal{D}'} \sum_{\substack{R = I \times J \in\mathcal{D}\textup{ good},\ R\subset S\\ \operatorname{gen}(R)=\operatorname{gen}(S)+(r,r)}}
|\langle u \rangle_S|^2 |\langle b, h_{I, \eta_1} \otimes u_{J, \kappa_1} \rangle|^2 \\
&= 2 \int_0^{\infty} \mathop{\sum_{S \in \mathcal{D}'}}_{|\langle u\rangle_S| > t} \sum_{\substack{R = I \times J \in\mathcal{D}\textup{ good},\ R\subset S\\ \operatorname{gen}(R)=\operatorname{gen}(S)+(r,r)}}|\langle b, h_{I, \eta_1} \otimes u_{J, \kappa_1} \rangle|^2 t\,dt \\
&\le 2 \int_0^{\infty} \mathop{\sum_{S \in \mathcal{D}'}}_{S \subset \{M_{\mathcal{D'}}u > t\}} \sum_{\substack{R = I \times J \in\mathcal{D}\textup{ good},\ R\subset S\\ \operatorname{gen}(R)=\operatorname{gen}(S)+(r,r)}}|\langle b, h_{I, \eta_1} \otimes u_{J, \kappa_1} \rangle|^2 t\,dt \\
&\le \|b\|_{\textup{BMO}_{\textup{prod}}(\mu)}^2 2 \int_0^{\infty} \mu(\{M_{\mathcal{D'}}u > t\})t\,dt \\
&= \|b\|_{\textup{BMO}_{\textup{prod}}(\mu)}^2 \|M_{\mathcal{D'}}u\|_{L^2(\mu)}^2 \\
&\le 16\|b\|_{\textup{BMO}_{\textup{prod}}(\mu)}^2 \|u\|_{L^2(\mu)}^2,
\end{align*}
which is what we wanted to prove.
\end{proof}

\section{Mixed paraproducts}\label{sec:mixed}

In the last few sections, we have completed the estimation of the part of the Haar series expansion of $\ave{Tf,g}$ involving quadruples of cubes with $\ell(I_1)\leq\ell(I_2)$ and $\ell(J_1)\leq\ell(J_2)$. The other three subseries with one or both of these inequalities reversed are completely symmetric, except for one detail: in the case that the two inequalities go in different directions, we end up with a different form of the full paraproduct, which we call mixed.

The following explanation reveals how it appears.
In the case $\ell(I_1) \le \ell(I_2)$ and $\ell(J_1) > \ell(J_2)$ one needs to deal with a full paraproduct, which is different in an essential way:
\begin{align*}
\Pi_{\textup{mixed},\,b}^{\eta_1\kappa_2}u
&:=  \sum_{I_2 \in \mathcal{D}_n'}  \mathop{\mathop{\sum_{I_1 \in \mathcal{D}_{n, \textup{good}}}}_{I_1 \subset I_2}}_{\ell(I_1) = 2^{-r}\ell(I_2)}
\sum_{J_1 \in \mathcal{D}_m} \mathop{\mathop{\sum_{J_2 \in \mathcal{D}_{m, \textup{good}}'}}_{J_2 \subset J_1}}_{\ell(J_2) = 2^{-r}\ell(J_1)}
\big\langle u, h_{I_1, \eta_1} \otimes \frac{\chi_{J_1}}{\mu_m(J_1)} \big\rangle \\
&\hspace{6.7cm} \times \langle b, h_{I_1, \eta_1} \otimes u_{J_2, \kappa_2} \rangle  \frac{\chi_{I_2}}{\mu_n(I_2)} \otimes u_{J_2, \kappa_2},
\end{align*}
$\eta_1 \ne 0$, $\kappa_2 \ne 0$. In the actual summation one has $b = T_1(1)$.
We need to separately demonstrate that $\|\Pi_{\textup{mixed},\,b}^{\eta_1\kappa_2}u\|_{L^2(\mu)} \lesssim \|b\|_{\textup{BMO}_{\textup{prod}}(\mu)} \|u\|_{L^2(\mu)}$.
Following \cite{PV}, we dualize with a $v\in L^2(\mu)$ and write
\begin{equation*}
  \ave{\Pi_{\textup{mixed},b}^{\eta_1\eta_2}u,v}=\pair{b}{\Lambda(u,v)}
\end{equation*}
for a suitable bilinear form $\Lambda$. We then need the two estimates
\begin{equation*}
  \abs{\ave{b,f}}\lesssim\|b\|_{\textup{BMO}_{\textup{prod}}(\mu)}\|f\|_{H^1(\mu)},\qquad
  \|\Lambda(u,v)\|_{H^1(\mu)}\lesssim\|u\|_{L^2(\mu)}\|v\|_{L^2(\mu)}
\end{equation*}
for a suitable (ad hoc) ``$H^1$''  norm $\|\cdot\|_{H^1(\mu)}$.

\subsection{$H^1$--BMO type duality inequality}\label{sec:dual}

We denote $\mathcal{C}_K:=\{I\in\mathcal{D}_n\textrm{ good}:I\subset K,\operatorname{gen}(I)=\operatorname{gen}(K)+r\}$ for $K\in\mathcal{D}_n'$, similarly $\mathcal{C}_L\subset\mathcal{D}_m$ for $L\in\mathcal{D}_m'$ and $\mathcal{C}_S:=\mathcal{C}_K\times\mathcal{C}_L\subset\mathcal{D}$ for $S=K\times L\in\mathcal{D}'=\mathcal{D}_n'\times\mathcal{D}_m'$. Also, we denote $S(I) = K$ if $I \in \mathcal{C}_K$.

\begin{defn}\label{defn:squaredef}
We define the square function
\begin{align*}
S_{\mathcal{D}_n\mathcal{D}_m}^{\mathcal{D}_n'\mathcal{D}_m'}f &= \Big(\mathop{\sum_{\eta_1 \ne 0}}_{\kappa_1 \ne 0} \mathop{\sum_{I_1 \in \mathcal{D}_{n, \textup{good}}}}_{J_1 \in \mathcal{D}_{m, \textup{good}}}
|\langle f, h_{I_1, \eta_1} \otimes u_{J_1, \kappa_1} \rangle|^2 \frac{\chi_{S(I_1)} \otimes \chi_{S(J_1)}}{\mu_n(S(I_1))\mu_m(S(J_1))}\Big)^{1/2} \\
&= \Big( \sum_{S\in\mathcal{D}'}\sum_{R\in\mathcal{C}_S}
|\langle f, h_R \rangle|^2 \frac{\chi_S}{\mu(S)}\Big)^{1/2},
\end{align*}
where we suppressed the $\eta_1, \kappa_1$ summation and denoted $h_R = h_{I_1, \eta_1} \otimes u_{J_1, \kappa_1}$.
\end{defn}
The following proposition is the main result of this section.
\begin{prop}\label{prop:dualprop}
If $b \in \BMO_{\textup{prod}}(\mu)$, then for all dyadic grids $\mathcal{D}_n$, $\mathcal{D}_n'$, $\mathcal{D}_m$ and $\mathcal{D}_m'$,
and for all functions $f$ of the form
\begin{displaymath}
f = \mathop{\sum_{\eta_1 \ne 0}}_{\kappa_1 \ne 0} \mathop{\sum_{I_1 \in \mathcal{D}_{n, \textup{good}}}}_{J_1 \in \mathcal{D}_{m, \textup{good}}} \lambda_{I_1J_1}^{\eta_1\kappa_1} h_{I_1, \eta_1} \otimes u_{J_1, \kappa_1}
= \sum_{S\in\mathcal{D}'}\sum_{R\in\mathcal{C}_S} \lambda_R h_R
\end{displaymath} 
there holds that
\begin{displaymath}
|\langle b, f\rangle| \lesssim \|b\|_{\textup{BMO}_{\textup{prod}}(\mu)} \|S_{\mathcal{D}_n\mathcal{D}_m}^{\mathcal{D}_n'\mathcal{D}_m'}f\|_{L^1(\mu)}.
\end{displaymath}
Here we use the same suppressed notation as in the above definition.
\end{prop}
\begin{proof}
We begin by simply estimating
\begin{displaymath}
|\langle b, f\rangle| \le \sum_{S\in\mathcal{D}'}\sum_{R\in\mathcal{C}_S} |\langle b, h_R\rangle \lambda_R|.
\end{displaymath}
Let
\begin{equation}\label{eq:phi}
  \phi:= S_{\mathcal{D}_n\mathcal{D}_m}^{\mathcal{D}_n'\mathcal{D}_m'}f =  \Big(\sum_{S\in\mathcal{D}'}\sum_{R\in\mathcal{C}_S}\abs{\lambda_R}^2\frac{\chi_S}{\mu(S)}\Big)^{1/2},
\end{equation}
and set $\Omega_k:=\{\phi>2^k\}$, and $\mathcal{S}_k:=\{S\in\mathcal{D}':\mu(S\cap\Omega_k)>\tfrac12\mu(S)\}$. Note that if $S\notin\bigcup_{k\in\Z}\mathcal{S}_k$, then $\mu(S\cap\{\phi=0\})\geq\tfrac12\mu(S)$, and so
\begin{equation*}
  \sum_{R\in\mathcal{C}_S}\abs{\lambda_R}^2
  \leq 2\sum_{R\in\mathcal{C}_S}\abs{\lambda_R}^2\frac{\mu(S\cap\{\phi=0\})}{\mu(S)}
  \leq 2\int_{\{\phi=0\}}\phi^2 \,d\mu = 0.
\end{equation*}
Therefore, all the relevant $S$ in the sum of interest belong to at least one $\mathcal{S}_k$. Also, $\bigcap_{k\in\Z}\Omega_k=\{\phi=\infty\}$, which has measure zero if the right side of the claim is finite (which we may assume).
This means that if $S \in \bigcap_{k\in\Z}\mathcal{S}_k$, then $\mu(S) = 0$ and so $\lambda_R = 0$ for every $R \in \mathcal{C}_S$.

We infer that for all the relevant $S$ in the sum of interest there exist a unique $k_S \in \Z$ so that $S \in \mathcal{S}_{k_S}\setminus\mathcal{S}_{k_S+1}$.
Thus, we may reorganize
\begin{equation*}
  \sum_{S\in\mathcal{D}'}\sum_{R\in\mathcal{C}_S}\abs{\pair{b}{h_R}\lambda_R}
  =\sum_{k\in\Z}  \sum_{S\in\mathcal{S}_k\setminus\mathcal{S}_{k+1}}\sum_{R\in\mathcal{C}_S}\abs{\pair{b}{h_R}\lambda_R}.
\end{equation*}
We then bound
\begin{equation}\label{eq:CSinSk}
\begin{split}
  \sum_{S\in\mathcal{S}_k\setminus\mathcal{S}_{k+1}}\sum_{R\in\mathcal{C}_S}\abs{\pair{b}{h_R}\lambda_R}
  &\leq\Big(\sum_{S\in\mathcal{S}_k\setminus\mathcal{S}_{k+1}}\sum_{R\in\mathcal{C}_S}\abs{\pair{b}{h_R}}^2\Big)^{1/2} \\
  &\qquad\times\Big(\sum_{S\in\mathcal{S}_k\setminus\mathcal{S}_{k+1}}\sum_{R\in\mathcal{C}_S}\abs{\lambda_R}^2\Big)^{1/2}.
\end{split}
\end{equation}
Clearly $S\in\mathcal{S}_k$ means that $S\subset\tilde\Omega_k:=\{M\chi_{\Omega_k}>\tfrac12\}$, where $M$ is the strong dyadic maximal operator with respect to $\mathcal{D}'$. Thus
\begin{equation*}
\begin{split}
  \Big(\sum_{S\in\mathcal{S}_k\setminus\mathcal{S}_{k+1}}\sum_{R\in\mathcal{C}_S}\abs{\pair{b}{h_R}}^2\Big)^{1/2}
 & \leq\Big(\sum_{S\subset\tilde\Omega_k}\sum_{R\in\mathcal{C}_S}\abs{\pair{b}{h_R}}^2\Big)^{1/2} \\
  & \leq\Norm{b}{\BMO_{\textup{prod}}(\mu)}\mu(\tilde\Omega_k)^{1/2}\lesssim\Norm{b}{\BMO_{\textup{prod}}(\mu)}\mu(\Omega_k)^{1/2}.
\end{split}
\end{equation*}
For the other factor in \eqref{eq:CSinSk}, we use the condition that $S\notin\mathcal{S}_{k+1}$ to write
\begin{equation*}
   1\leq 2\mu(S\setminus\Omega_{k+1})/\mu(S), 
\end{equation*}
and again the condition that $S\in\mathcal{S}_k$ to see that $S\subset\tilde\Omega_k$,
so that
\begin{equation*}
\begin{split}
  \sum_{S\in\mathcal{S}_k\setminus\mathcal{S}_{k+1}}\sum_{R\in\mathcal{C}_S}\abs{\lambda_R}^2
  &\leq 2\int_{\tilde\Omega_k\setminus\Omega_{k+1}} 
      \sum_{S\in\mathcal{S}_k\setminus\mathcal{S}_{k+1}} \sum_{R\in\mathcal{C}_S}\abs{\lambda_R}^2\frac{\chi_S}{\mu(S)}\ud\mu \\
  &\leq 2\int_{\tilde\Omega_k\setminus\Omega_{k+1}}\phi^2\ud\mu
  \leq 2\cdot (2^{k+1})^2\mu(\tilde\Omega_k)\lesssim 2^{2k}\mu(\Omega_k).
\end{split}
\end{equation*}
Substituting back to \eqref{eq:CSinSk}, we obtain
\begin{equation*}
  \sum_{S\in\mathcal{S}_k\setminus\mathcal{S}_{k+1}}\sum_{R\in\mathcal{C}_S}\abs{\pair{b}{h_R}\lambda_R}
  \lesssim\Norm{b}{\BMO_{\textup{prod}}(\mu)}\mu(\Omega_k)^{1/2}\cdot 2^k\mu(\Omega_k)^{1/2},
\end{equation*}
and hence
\begin{equation*}
  \sum_{S\in\mathcal{D}'}\sum_{R\in\mathcal{C}_S}\abs{\pair{b}{h_R}\lambda_R}
  \lesssim\sum_{k\in\Z}\Norm{b}{\BMO_{\textup{prod}}(\mu)} 2^k\mu(\Omega_k)\lesssim\Norm{b}{\BMO_{\textup{prod}}(\mu)}\Norm{\phi}{L^1(\mu)}.\qedhere
\end{equation*}
\end{proof}

\subsection{Boundedness of the mixed paraproduct}
We are now ready to prove:


\begin{prop}
There holds that $\|\Pi_{\textup{mixed},\,b}^{\eta_1\kappa_2}u\|_{L^2(\mu)} \lesssim \|b\|_{\textup{BMO}_{\textup{prod}}(\mu)}\|u\|_{L^2(\mu)}$.
\end{prop}

\begin{proof}
This proof is adapted from \cite{PV}.

Let us abbreviate $L:=\|b\|_{\textup{BMO}_{\textup{prod}}(\mu)}$. We will show that
\begin{equation*}
  |\langle \Pi_{\textup{mixed},\,b}^{\eta_1\kappa_2}u, v\rangle| \lesssim  L\|u\|_{L^2(\mu)} \|v\|_{L^2(\mu)}.
\end{equation*}
 We write
\begin{align*}
\langle \Pi_{\textup{mixed},\,b}^{\eta_1\kappa_2}u, v\rangle = \Big\langle b, \sum_{I_2 \in \mathcal{D}_n'}  \mathop{\mathop{\sum_{I_1 \in \mathcal{D}_{n, \textup{good}}}}_{I_1 \subset I_2}}_{\ell(I_1) = 2^{-r}\ell(I_2)}
\sum_{J_1 \in \mathcal{D}_m}& \mathop{\mathop{\sum_{J_2 \in \mathcal{D}_{m, \textup{good}}'}}_{J_2 \subset J_1}}_{\ell(J_2) = 2^{-r}\ell(J_1)}
\big\langle u, h_{I_1, \eta_1} \otimes \frac{\chi_{J_1}}{\mu_m(J_1)}\big\rangle \\
&\times  \big\langle v,  \frac{\chi_{I_2}}{\mu_n(I_2)} \otimes u_{J_2, \kappa_2} \big\rangle h_{I_1, \eta_1} \otimes u_{J_2, \kappa_2} \Big\rangle.
\end{align*}
By  Proposition~\ref{prop:dualprop}, we have
\begin{displaymath}
|\langle \Pi_{\textup{mixed},\,b}^{\eta_1\kappa_2}u, v\rangle| \lesssim L\|S_{\mathcal{D}_n\mathcal{D}_m'}^{\mathcal{D}_n'\mathcal{D}_m}f\|_{L^1(\mu)},
\end{displaymath}
where
\begin{displaymath}
f = \sum_{I_1 \in \mathcal{D}_{n, \textup{good}}} \sum_{J_2 \in \mathcal{D}_{m, \textup{good}}'}  \big\langle u, h_{I_1, \eta_1} \otimes \frac{\chi_{S(J_2)}}{\mu_m(S(J_2))}\big\rangle
\big\langle v,  \frac{\chi_{S(I_1)}}{\mu_n(S(I_1))} \otimes u_{J_2, \kappa_2} \big\rangle h_{I_1, \eta_1} \otimes u_{J_2, \kappa_2}.
\end{displaymath}

We estimate $S_{\mathcal{D}_n\mathcal{D}_m'}^{\mathcal{D}_n'\mathcal{D}_m}f$ pointwise. Let us first note that
\begin{align*}
(S_{\mathcal{D}_n\mathcal{D}_m'}^{\mathcal{D}_n'\mathcal{D}_m}f)^2
\le  \sum_{I_1 \in \mathcal{D}_{n, \textup{good}}} \sum_{J_2 \in \mathcal{D}_{m, \textup{good}}'} \big| \big\langle &u, h_{I_1, \eta_1} \otimes \frac{\chi_{S(J_2)}}{\mu_m(S(J_2))}\big\rangle \big|^2 \\
& \times \big| \big\langle v,  \frac{\chi_{S(I_1)}}{\mu_n(S(I_1))} \otimes u_{J_2, \kappa_2} \big\rangle \big|^2 \frac{\chi_{S(I_1)} \otimes \chi_{S(J_2)}}{\mu_n(S(I_1))\mu_m(S(J_2))}.
\end{align*}
This implies that
\begin{align*}
S_{\mathcal{D}_n\mathcal{D}_m'}^{\mathcal{D}_n'\mathcal{D}_m}f \le \Big(& \sum_{I_1 \in \mathcal{D}_{n, \textup{good}}} \frac{\chi_{S(I_1)}}{\mu_n(S(I_1))} \otimes
\sup_{J_2 \in \mathcal{D}_{m, \textup{good}}'}\big[ \big|\big\langle u, h_{I_1, \eta_1} \otimes \frac{\chi_{S(J_2)}}{\mu_m(S(J_2))}\big\rangle \big|^2  \chi_{S(J_2)}\big]\Big)^{1/2} \\
&\times \Big( \sum_{J_2 \in \mathcal{D}_{m, \textup{good}}'} \sup_{I_1 \in \mathcal{D}_{n, \textup{good}}}\big[ \big| \big\langle v,  \frac{\chi_{S(I_1)}}{\mu_n(S(I_1))} \otimes u_{J_2, \kappa_2} \big\rangle \big|^2
\chi_{S(I_1)} \big] \otimes \frac{\chi_{S(J_2)}}{\mu_m(S(J_2))} \Big)^{1/2}.
\end{align*}
This yields the final pointwise bound
\begin{align*}
S_{\mathcal{D}_n\mathcal{D}_m'}^{\mathcal{D}_n'\mathcal{D}_m}f \le \Big(& \sum_{I_1 \in \mathcal{D}_{n, \textup{good}}} \frac{\chi_{S(I_1)}}{\mu_n(S(I_1))} \otimes M_{\mathcal{D}_m}(\langle u, h_{I_1, \eta_1}\rangle_1)^2 \Big)^{1/2} \\
&\times \Big( \sum_{J_2 \in \mathcal{D}_{m, \textup{good}}'} M_{\mathcal{D}_n'}(\langle v, u_{J_2, \kappa_2}\rangle_2)^2  \otimes \frac{\chi_{S(J_2)}}{\mu_m(S(J_2))} \Big)^{1/2}.
\end{align*}

Using the above pointwise bound together with Cauchy--Schwarz we get that
\begin{align*}
\|S_{\mathcal{D}_n\mathcal{D}_m'}^{\mathcal{D}_n'\mathcal{D}_m}f\|_{L^1(\mu)} \le \Big\|  \sum_{I_1 \in \mathcal{D}_{n, \textup{good}}}& \frac{\chi_{S(I_1)}}{\mu_n(S(I_1))} \otimes M_{\mathcal{D}_m}(\langle u, h_{I_1, \eta_1}\rangle_1)^2 \Big\|_{L^1(\mu)}^{1/2} \\
&\times \Big\| \sum_{J_2 \in \mathcal{D}_{m, \textup{good}}'} M_{\mathcal{D}_n'}(\langle v, u_{J_2, \kappa_2}\rangle_2)^2  \otimes \frac{\chi_{S(J_2)}}{\mu_m(S(J_2))} \Big\|_{L^1(\mu)}^{1/2}.
\end{align*}
The proof is ended by noting that
\begin{align*}
&\Big\|  \sum_{I_1 \in \mathcal{D}_{n, \textup{good}}} \frac{\chi_{S(I_1)}}{\mu_n(S(I_1))} \otimes M_{\mathcal{D}_m}(\langle u, h_{I_1, \eta_1}\rangle_1)^2 \Big\|_{L^1(\mu)}\\
&= \sum_{I_1 \in \mathcal{D}_{n, \textup{good}}} \int_{\R^m} M_{\mathcal{D}_m}(\langle u, h_{I_1, \eta_1}\rangle_1)^2\,d\mu_m \\
&\lesssim \sum_{I_1 \in \mathcal{D}_n} \int_{\R^m} |\langle u, h_{I_1, \eta_1}\rangle_1|^2\,d\mu_m \lesssim \|u\|_{L^2(\mu)}^2.
\end{align*}
\end{proof}

\section{Necessity of $T1\in\textup{BMO}_{\textup{prod}}(\mu)$}\label{sec:bmo}

\subsection{Dyadic Journ\'e's lemma for general product measures}\label{sec:journe}
We prove Journ\'e's covering lemma \cite{Jo1} with general measures.
In this section, we assume that $\mu=\mu_n\times\mu_m$ is an arbitrary product of a Borel measure $\mu_n$ in $\R^n$ and $\mu_m$ in $\R^m$. That is, we don't assume any growth conditions.
Again, we have the dyadic systems $\mathcal{D}_n$ and $\mathcal{D}_m$ in $\R^n$ and $\R^m$ respectively.

For $R=I\times J\in\mathcal{D}:=\mathcal{D}_n \times\mathcal{D}_m$, write $R^{(i,j)}:=I^{(i)}\times J^{(j)}$ and
\begin{equation*}
  \operatorname{gen}_1(R):=\operatorname{gen}(I), \qquad
  \operatorname{gen}(R):=(\operatorname{gen}(I),\operatorname{gen}(J)), 
\end{equation*}
where $I^{(i)}$ is the usual dyadic ancestor, and $\operatorname{gen}(I)$ is the usual generation of a dyadic cube.

Let $\Omega \subset \R^{n+m}$ be such a set that $\mu(\Omega) < \infty$ and that for every $x \in \Omega$ there
exists $R \in \mathcal{D}$ so that $x \in R \subset \Omega$. We let
\begin{equation*}
  \tilde\Omega:=\{M\chi_\Omega>\tfrac12\},
\end{equation*}
where $M$ is the strong dyadic maximal operator with respect to $\mathcal{D}$.
Define the embeddedness of $R$ in $\Omega$ as
\begin{equation*}
  \operatorname{emb}_1(R;\Omega):=\sup\{k:R^{(k,0)}\subset\tilde\Omega\}.
\end{equation*}
This notation and some related inspiration is derived from \cite{CLMP}.

We say that a dyadic rectangle $R=I\times J\subset\Omega$ is 2-maximal if $I\times\tilde{J}\not\subset\Omega$ for any dyadic $\tilde{J}\supsetneq J$. (Obviously we could define $\operatorname{emb}_2(R;\Omega)$ and $1$-maximality analogously. The main point is that in the following result we need to consider embeddedness and maximality with respect to different variables.)

\begin{thm}[Journ\'e's lemma]\label{thm:Journe}
Let $\omega:\N\to\R_+$ be a decreasing function with the property that $\sum_{k=0}^{\infty} \omega(k) < \infty$.
Then there holds that
\begin{equation*}
  \sum_{\substack{R\subset\Omega\\ \textup{2-maximal} }}\omega(\operatorname{emb}_1(R;\Omega))\times\mu(R)
  \leq 2\sum_{k=0}^\infty\omega(k)\times\mu(\Omega).
\end{equation*}
\end{thm}

\begin{proof}
Let us define $\delta(k):=\omega(k)-\omega(k+1)\geq 0$, so that $\omega(j)=\sum_{k=j}^\infty\delta(k)$.
We have
\begin{equation}\label{eq:JourneStart}
\begin{split}
   \sum_{\substack{R\subset\Omega\\ \textup{2-maximal}}} &\omega(\operatorname{emb}_1(R;\Omega))\mu(R)
    =\sum_{j=0}^\infty \omega(j)\sum_{\substack{R\subset\Omega\textup{ 2-maximal}\\ \operatorname{emb}_1(R;\Omega)=j}}\mu(R) \\
    &=\sum_{k=0}^\infty \delta(k)\sum_{\substack{R\subset\Omega\textup{ 2-maximal}\\ \operatorname{emb}_1(R;\Omega)\leq k}}\mu(R) 
     =\sum_{k=0}^\infty \delta(k) \sum_{i=0}^{k}
      \sum_{R\in\mathcal{R}(k,i)}\mu(R),
\end{split}
\end{equation}
where
\begin{equation*}
  \mathcal{R}(k,i):=\Big\{R\subset\Omega\quad\textup{2-maximal};\
    \operatorname{emb}_1(R;\Omega)\leq k,\
    \operatorname{gen}_1(R)\equiv i\mod k+1\Big\}.
\end{equation*}

We make the following observation about two intersecting $R=I\times J$ and $R'=I'\times J'$ in $\mathcal{R}(k,i)$. By definition, we have $I\cap I'\neq\varnothing$ and $J\cap J'\neq\varnothing$. Suppose first that $I=I'$. But then also $J=J'$, since otherwise one of $R$ and $R'$ could not be 2-maximal. Thus $I\neq I'$ whenever $R\neq R'$. Let us then consider the case that $I\subsetneq I'$. Then we must have $J\supseteq J'$, since otherwise $R=I\times J\subsetneq I\times J'\subset R'\subset\Omega$ would contradict the 2-maximality of $R$.

For every $R=I\times J\in\mathcal{R}(k,i)$, consider the subset
\begin{equation*}
   E(R):=R\setminus\bigcup_{\substack{R'=I'\times J'\in\mathcal{R}(k,i)\\ I'\supsetneq I}}R'
     =I\times\Big(J\setminus \bigcup_{\substack{I'\times J'\in\mathcal{R}(k,i)\\ I'\supsetneq I}}J'\Big),
\end{equation*}
where we observe that $J'\subset J$ for all relevant $J'$ in the union (i.e., those for which $R'$ intersects $R$), by what we just checked above.
The sets $E(R)$, $R\in\mathcal{R}(k,i)$, are pairwise disjoint; namely, if two different $R=I\times J,R'=I'\times J'\in\mathcal{R}(k,i)$ intersect, the previous paragraph implies that either $I'\supsetneq I$, in which case $R'\cap E(R)=\varnothing$, or else $I\supsetneq I'$, in which case $R\cap E(R')=\varnothing$.

We claim that
\begin{equation}\label{eq:JourneClaim}
  \mu(E(R))=\mu_n(I)\mu_m\Big( J\setminus \bigcup_{\substack{I'\times J'\in\mathcal{R}(k,i)\\ I'\supsetneq I}}J'\Big)
  \geq \frac12\mu(R)=\frac12\mu_n(I)\mu_m(J).
\end{equation}
This is trivial if $\mu_n(I)=0$. Otherwise, suppose that the opposite estimate is valid.
Observe that $I\subsetneq I'$ and $\operatorname{gen}(I)\equiv i\equiv\operatorname{gen}(I')\mod k+1$ imply that $I^{(k+1)}\subset I'$. Thus
\begin{equation*}
  R^{(k+1,0)}\setminus\Omega
  \subset R^{(k+1,0)}\setminus\bigcup_{\substack{R'=I'\times J'\in\mathcal{R}(k,i)\\ I'\supsetneq I}}R'
  =I^{(k+1)}\times\Big(J\setminus \bigcup_{\substack{I'\times J'\in\mathcal{R}(k,i)\\ I'\supsetneq I}}J'\Big).
\end{equation*}
Hence
\begin{equation*}
\begin{split}
  \mu(R^{(k+1,0)}\setminus\Omega)
  &\leq\mu_n(I^{(k+1)})\mu_m\Big(J\setminus \bigcup_{\substack{I'\times J'\in\mathcal{R}(k,i)\\ I'\supsetneq I}}J'\Big) \\
  &<\mu_n(I^{(k+1)})\cdot\frac12\mu_m(J)=\frac12\mu(R^{(k+1,0)})
\end{split}
\end{equation*}
and therefore
\begin{equation*}
  \inf_{R^{(k+1,0)}}M\chi_\Omega
  \geq\frac{\mu(R^{(k+1,0)}\cap\Omega)}{\mu(R^{(k+1,0)})}
  >1-\frac12=\frac12.
\end{equation*}
But this means that $R^{(k+1,0)}\subset\tilde\Omega$, contradicting $\operatorname{emb}_1(R;\Omega)\leq k$.

Hence the claim \eqref{eq:JourneClaim} is valid. But this means that
\begin{equation*}
  \sum_{R\in\mathcal{R}(k,i)}\mu(R)
  \leq \sum_{R\in\mathcal{R}(k,i)}2\mu (E(R))\leq 2\mu(\Omega),
\end{equation*}
where the last step is due to the fact that the sets $E(R)$ are pairwise disjoint and all contained in $\Omega$. Substituting back to \eqref{eq:JourneStart} shows that
\begin{equation*}
     \sum_{\substack{R\subset\Omega\\ \textup{2-maximal}}} \omega(\operatorname{emb}_1(R;\Omega))\mu(R)
     \leq\sum_{k=0}^\infty\delta(k)\times(k+1)\times 2\mu(\Omega)
     =2\sum_{k=0}^\infty\omega(k)\times\mu(\Omega).\qedhere
\end{equation*}
\end{proof}

\subsection{
A class of bi-parameter singular integrals}
In this section we return to the setting of upper doubling measures. Let $\mu = \mu_n \times \mu_m$, where $\mu_n$ and $\mu_m$ are upper doubling measures on $\R^n$ and $\R^m$ respectively. The corresponding dominating functions
are denoted by $\lambda_n$ and $\lambda_m$. We consider what seems like a slightly restricted class of product Calder\'on--Zygmund operators than above. In fact, the definition will be close to the classical setting of Journ\'e, just with general measures.
We want to demonstrate the necessity of the product BMO condition for this subclass.

We consider product Calder\'on--Zygmund operators $T\in\bddlin(L^2(\mu))$ with the following vector-valued Calder\'on--Zygmund structure. When viewing $T$ as an operator on $L^2(\mu_n;L^2(\mu_m))$, it has a kernel $T_1(x_1,y_1)$ with values in $L^2(\mu_m)$-bounded Calder\'on--Zygmund operators, such that
\begin{equation*}
  \pair{\Phi}{T\Psi}
  =\iint\pair{\Phi(x_1)}{T_1(x_1,y_1)\Psi(y_1)}_2\ud\mu_n(x_1)\ud\mu_n(y_1)
\end{equation*}
whenever the $L^2(\mu_m)$-valued functions $\Phi$ and $\Psi$ are disjointly supported. This kernel is required to satisfy the standard estimates
\begin{equation*}
  \Norm{T_1(x_1,y_1)}{CZ(\mu_m)}\lesssim \frac{1}{\lambda_n(x_1,|x_1-y_1|)}
\end{equation*}
and
\begin{equation*}
  \Norm{T_1(x_1,y_1)-T_1(x_1',y_1)}{CZ(\mu_m)}\lesssim\Big(\frac{|x_1-x_1'|}{|x_1-y_1|}\Big)^\alpha \frac{1}{\lambda_n(x_1,|x_1-y_1|)}
\end{equation*}
for $|x_1-y_1|>2|x_1-x_1'|$. (We do not require regularity in the second variable for the present considerations.) Here $\Norm{\ }{CZ(\mu_m)}$ designates the sum of the $\bddlin(L^2(\mu_m))$-norm of the operator,
and the Calder\'on--Zygmund constants of its (scalar-valued) kernel.
We impose the analogous conditions when viewing $T$ as an operator on the space $L^2(\mu_m;L^2(\mu_n))$.
\begin{rem}
Let $\Phi = \Phi_1 \otimes \Phi_2$ and $\Psi = \Psi_1 \otimes \Psi_2$ with $\textup{spt}\,\Phi_1 \cap \textup{spt}\,\Psi_1 = \emptyset$. Then with our original notation we have
$K_{\Phi_2, \Psi_2}(x_1, y_1) = \langle \Phi_2, T_1(x_1,y_1)\Psi_2\rangle_2$, and we have the Calder\'on--Zygmund bounds
with the constant $C(\Phi_2, \Psi_2) \lesssim \|\Phi_2\|_{L^2(\mu_m)} \|\Psi_2\|_{L^2(\mu_m)}$. If also $\textup{spt}\,\Phi_2 \cap \textup{spt}\, \Psi_2 = \emptyset$, then one has
a representation with the full kernel $K(x, y) = K_{T_1(x_1,y_1)}(x_2,y_2)$. We will show that for this subclass of operators
our new non-homogeneous product BMO condition is also necessary. Although, let us note again that for example the regularity in the second variable is not needed in this section.
We use this subclass, since it seems that it allows more control in situations where not all of the functions are of tensor product type, and this type of control is needed here for the first time.
\end{rem}

\begin{thm}\label{thm:BMO}
Let $T$ be an $L^2(\mu)$ bounded product Calder\'on--Zygmund operator in the sense above, and $b\in L^\infty(\mu)$. Then
\begin{equation*}
  \|Tb\|_{\textup{BMO}_{\textup{prod}}(\mu)}\lesssim\|b\|_{L^{\infty}(\mu)}.
\end{equation*}
\end{thm}

\subsection{Initial considerations}
 We may assume by homogeneity that $\|b\|_{L^{\infty}(\mu)}=1$. By Definition~\ref{def:BMOprod} and Remark~\ref{rem:BMOprod}, we need to check the estimate
\begin{equation*}
  \sum_{\substack{S\in\mathcal{D}'\\S\subset\Omega}}
  \sum_{\substack{R\in\mathcal{D}\textup{ good},\ R\subset S\\ \operatorname{gen}(R)=\operatorname{gen}(S)+(r,r)}}
    \abs{\pair{h_R}{Tb}}^2
  \lesssim\mu(\Omega)
\end{equation*}
for any dyadic systems $\mathcal{D}=\mathcal{D}_n\times\mathcal{D}_m$ and $\mathcal{D}'=\mathcal{D}_n'\times\mathcal{D}_m'$, where $\Omega \subset \R^{n+m}$ is a bounded set such that $\mu(\Omega) < \infty$ and that for every $x \in \Omega$ there exists $S \in \mathcal{D}'$ so that $x \in S \subset \Omega$.  We have also used the following short-hand notation: If $R = I \times J$, then $h_R = h_{I \times J} = h_I \otimes h_J$ is shorthand for $h_{I,\eta} \otimes u_{J,\kappa}$ for $\eta \ne 0$ and $\kappa \ne 0$ (that is, only cancellative Haar functions appear).

We start by setting up some further simplifying notation. For $K\in\mathcal{D}_n'$, let
\begin{equation*}
  \mathcal{C}_K:=\{I\in\mathcal{D}_n\textup{ good}:I\subset K,\operatorname{gen}(I)=\operatorname{gen}(K)+r\}
\end{equation*}
and
\begin{equation*}
  \hat{K}:=\bigcup_{I\in\mathcal{C}_K}I \subset \{x_1\in K:d(x_1,K^c)\geq 2^{-r\gamma_n}\ell(K)\},
\end{equation*}
where the last containment is immediate from the definition of goodness.
We use similar notation $\mathcal{C}_L$ and $\hat{L}$ for $L\in\mathcal{D}_m'$.

We decompose $b=b\chi_{\tilde\Omega}+b\chi_{\tilde\Omega^c}$, where $\tilde\Omega:=\{M_{\mathcal{D}'}\chi_\Omega>\tfrac12\}$ and $M_{\mathcal{D}'}$ is the strong dyadic maximal operator with respect to $\mathcal{D}'$. Then
\begin{equation*}
\begin{split}
  \sum_{S=K\times L\subset\Omega}\sum_{R\in\mathcal{C}_K\times\mathcal{C}_L}\abs{\pair{h_R}{T(b\chi_{\tilde\Omega})}}^2
  &\leq\sum_{R\in\mathcal{D}}\abs{\pair{h_R}{T(b\chi_{\tilde\Omega})}}^2
  \leq\Norm{T(b\chi_{\tilde\Omega})}{2}^2 \\
  &\lesssim\Norm{b\chi_{\tilde\Omega}}{2}^2\lesssim\mu(\tilde\Omega)\lesssim\mu(\Omega),
\end{split}
\end{equation*}
using
\begin{equation*}
  \mu(\tilde\Omega)
  =\mu(M_{\mathcal{D}'}\chi_\Omega>\tfrac12)
  \leq 4\Norm{M_{\mathcal{D}'}\chi_\Omega}{2}^2
  \lesssim\Norm{\chi_\Omega}{2}^2=\mu(\Omega)
\end{equation*}
in the last step.

So it suffices to prove the estimate with $b\chi_{\tilde\Omega^c}$ in place of $b$. For simplicity, we denote it again by $b$, assuming from now on that $\supp b\subset\tilde\Omega^c$.

For every $L\in\mathcal{D}'_m$, let $\mathcal{F}_L$ be the collection of the maximal $F\in\mathcal{D}'_n$ such that $F\times L\subset\tilde\Omega$, and denote $F_L:=\bigcup_{F\in\mathcal{F}_L}F$. If there are any cubes $F\in\mathcal{D}'_n$ with $F\times L\subset\tilde\Omega$, then the maximal cubes exist by standard properties of dyadic cubes, except if there is a strictly increasing sequence of cubes $F_k\in\mathcal{D}'_n$ with $F_k\times L\subset\tilde\Omega$. Since our dyadic systems do not have quadrants, such a sequence will exhaust all of $\R^n$, and then in fact $\R^n\times L\subset\tilde\Omega$. In this case we set $\mathcal{F}_L:=\{\R^n\}$, $F_L:=\R^n$. We refer to this as the \emph{degenerate case}.

These notions lead to the splitting $b=b^1_L+b^2_L$, where
\begin{equation*}
  b^1_L(x_1,x_2):=b(x_1,x_2)\chi_{F_L}(x_1).
\end{equation*}
Note that if $b^1_L(y_1,y_2)\neq 0$, then $y_1\in F_L$ while $(y_1,y_2)\in\tilde\Omega^c\subset(F_L\times L)^c$, and hence $y_2\in L^c$. Thus $\textup{spt}\,b^1_L\subset F_L\times L^c$. Note that $b^1_L=b$, $b^2_L=0$ in the degenerate case.

Now we estimate the left hand side in the Theorem by
\begin{equation*}
   \sum_{i=1}^2
        \sum_{\substack{S=K\times L\in\mathcal{D}'\\ S\subset\Omega}}
        \sum_{\substack{I\in\mathcal{C}_K\\ J\in\mathcal{C}_L}}
    \abs{\pair{h_{I\times J}}{Tb_L^i}}^2.
\end{equation*}

\subsection{Analysis of $b^1_L$}

We do not need to pay any special attention to the possible degenerate cases in this analysis. The set $F_L$ will appear every now and then, but it can equally well be the full space $\R^n$ or a subset thereof. The only property that we need is that $\mu_n(F_L)<\infty$ whenever this set appears. To see this, observe that $b^1_L$ only appears in the pairing $\ave{h_{I\times J},Tb^1_L}$, where $J\in \mathcal{C}_L$. If $\mu_m(L)=0$, then $h_{I\times J}=0$ for all $J\in\mathcal{C}_L$, and we can ignore such pairings. On the other hand, if $\mu_m(L)>0$, then
\begin{equation*}
  \mu_n(F_L)\mu_m(L)=\mu(F_L\times L)\leq\mu(\tilde\Omega)<\infty
\end{equation*}
implies that $\mu_n(F_L)<\infty$, as claimed.

We begin with:

\begin{lem}
For a fixed $L \in \mathcal{D}_m'$, the sum over the other variables is estimated as
\begin{equation*}
\begin{split}
   \sum_{K\in\mathcal{D}'_n}
    &\sum_{\substack{I\in\mathcal{C}_K\\ J\in\mathcal{C}_L}}
   \abs{\pair{h_{I\times J}}{Tb_L^1}}^2 \\
  & \lesssim \mu_m(\hat{L})\int_{\R^n} \int_{\supp b^1_L(y_1,\cdot)}
   \Big(\frac{\ell(L)}{d(\hat L,y_2)}\Big)^\beta
   \frac{d\mu_m(y_2)}{\lambda_m(x_{\hat L},d(\hat L,y_2))}\,d\mu_n(y_1),
\end{split}
\end{equation*}
where $x_{\hat L}$ is an arbitrarily fixed point in $\hat{L}$.
\end{lem}

Note that we have dropped the restriction that $K\times L\subset\Omega$; this part of the estimate is true even if we allow $K$ to range over all dyadic intervals.
\begin{proof}
We can write
\begin{equation*}
\begin{split}
  \pair{h_{I\times J}}{Tb^1_L}
  &=\iint h_J(x_2)\pair{h_I}{T_2(x_2,y_2)b^1_L(\cdot,y_2)}_1\ud\mu_m(x_2)\ud\mu_m(y_2)  \\
  &=\iint h_J(x_2)\pair{h_I}{[T_2(x_2,y_2)-T_2(c_J,y_2)]b^1_L(\cdot,y_2)}_1\ud\mu_m(x_2)\ud\mu_m(y_2),  
\end{split}
\end{equation*}
where $T_2(x_2,y_2)$ is the $\bddlin(L^2(\mu_n))$-valued kernel of $T$ when interpreted as an operator on $L^2(\mu_m;L^2(\mu_n))$; it acts on the function $b^1_L(\cdot,y_2)\in L^2(\mu_n)$ (the function is bounded and supported on $F_L$), and $\pair{\ }{\ }_1$ designates the inner product in $L^2(\mu_n)$. Hence, abbreviating
\begin{equation*}
   B_J(x_2,y_2):=[T_2(x_2,y_2)-T_2(c_J,y_2)]b^1_L(\cdot,y_2)\in L^2(\mu_n), 
\end{equation*}
we have that
\begin{displaymath}
\pair{h_{I\times J}}{Tb^1_L} = \iint h_J(x_2)\pair{h_I}{B_J(x_2,y_2)}_1\ud\mu_m(x_2)\ud\mu_m(y_2).
\end{displaymath}
Thus, by the orthonormality of the Haar functions $h_I\in L^2(\mu_n)$, we have (with a fixed $J \in \mathcal{C}_L$) that
\begin{align*}
\sum_{\substack{K\in\mathcal{D}_n'\\ I\in\mathcal{C}_K}} \abs{\pair{h_{I\times J}}{Tb^1_L}}^2
&\le \sum_{I\in\mathcal{D}_n} \abs{\pair{h_{I\times J}}{Tb^1_L}}^2 \\
&\le \Big[\iint\abs{h_J(x_2)}\Big(\sum_{I\in\mathcal{D}_n}\abs{\pair{h_I}{B_J(x_2,y_2)}_1}^2\Big)^{1/2}\ud\mu_m(x_2)\ud\mu_m(y_2)\Big]^2 \\
&\le \Big[ \iint |h_J(x_2)| \|B_J(x_2,y_2)\|_{L^2(\mu_n)} \ud\mu_m(x_2)\ud\mu_m(y_2) \Big]^2.
\end{align*}


Since $x_2 \in J$, we have $|x_2-c_J| \le \ell(J)/2$. Moreover, since $y_2 \in L^c$, there holds that
\begin{displaymath}
|c_J-y_2| \ge d(J,L^c) \ge \ell(J)^{\gamma_m}\ell(L)^{1-\gamma_m} \ge \ell(J) \ge 2|x_2-c_J|.
\end{displaymath}
Hence the H\"older estimate for $T_2$ gives that
\begin{equation*}
   \Norm{T_2(x_2,y_2)-T_2(c_J,y_2))}{\bddlin(L^2(\mu_n))}\lesssim\Big(\frac{\ell(J)}{|c_J-y_2|}\Big)^\beta\frac{1}{\lambda_m(c_J,|c_J-y_2|)}.
\end{equation*}
It follows that
\begin{equation*}
\begin{split}
  \sum_{\substack{K\in\mathcal{D}_n'\\ I\in\mathcal{C}_K\\ J\in\mathcal{C}_L}}
   &\abs{\pair{h_{I\times J}}{Tb^1_L}}^2  \\
   &\lesssim \sum_{J\in\mathcal{C}_L}\Big[\int\abs{h_J(x_2)}\int\Big(\frac{\ell(J)}{|c_J-y_2|}\Big)^\beta
   \frac{\Norm{b^1_L(\cdot,y_2)}{L^2(\mu_n)}}{\lambda_m(c_J,|c_J-y_2|)}\ud\mu_m(y_2)\ud\mu_m(x_2)\Big]^2 \\
   &\lesssim \sum_{J\in\mathcal{C}_L} \mu_m(J)
   \Big[ \int\Big(\frac{\ell(J)}{|c_J-y_2|}\Big)^\beta
   \frac{\Norm{b^1_L(\cdot,y_2)}{L^2(\mu_n)}}{\lambda_m(c_J,|c_J-y_2|)}\ud\mu_m(y_2) \Big]^2.
\end{split}
\end{equation*}
By Cauchy--Schwarz, the last brackets squared are dominated by
\begin{align*}
\Big[ \int_{J^c} \Big(&\frac{\ell(J)}{|c_J-y_2|}\Big)^\beta
   \frac{1}{\lambda_m(c_J,|c_J-y_2|)}\ud\mu_m(y_2)\Big] \\ &\times \Big[ \int \Big(\frac{\ell(J)}{|c_J-y_2|}\Big)^\beta
   \frac{\Norm{b^1_L(\cdot,y_2)}{L^2(\mu_n)}^2}{\lambda_m(c_J,|c_J-y_2|)}\ud\mu_m(y_2)\Big].
\end{align*}
It is a standard fact that
\begin{displaymath}
\int_{J^c} \Big(\frac{\ell(J)}{|c_J-y_2|}\Big)^\beta \frac{1}{\lambda_m(c_J,|c_J-y_2|)}\ud\mu_m(y_2) \lesssim 1.
\end{displaymath}

To estimate the other bracket we need a few observations.
Since $c_J \in J \in \mathcal{C}_L$, we have that $c_J \in \hat L$ and so $|c_J - y_2| \ge d(\hat L, y_2)$. So we may estimate
\begin{displaymath}
\Big(\frac{\ell(J)}{|c_J-y_2|}\Big)^\beta \frac{1}{\lambda_m(c_J,|c_J-y_2|)} \le \Big(\frac{\ell(L)}{d(\hat L, y_2)}\Big)^\beta \frac{1}{\lambda_m(c_J,d(\hat L, y_2))}.
\end{displaymath}
We still want to get rid of the last dependence on $J$, the $c_J$ on $\lambda_m(c_J,d(\hat L, y_2))$. To this end, let $x_{\hat L}$ be an arbitrarily fixed point in $\hat{L}$ as in the statement of the Lemma.
Since $c_J, x_{\hat L} \in L$, we have $|c_J - x_{\hat L}| \le \ell(L)$. Moreover, we have $d(\hat L, y_2) \ge d(\hat L, L^c) \ge 2^{-r\gamma_m} \ell(L) \ge 2^{-r\gamma_m}|c_J - x_{\hat L}|$. Therefore, we can estimate that
\begin{displaymath}
\lambda_m(c_J,d(\hat L, y_2)) \sim \lambda_m(c_J, 2^{r\gamma_m}d(\hat L, y_2)) \sim \lambda_m(x_{\hat L}, 2^{r\gamma_m}d(\hat L, y_2)) \sim \lambda_m(x_{\hat L}, d(\hat L, y_2)).
\end{displaymath}
So for the other bracket we get the bound
\begin{align*}
\int \Big(&\frac{\ell(J)}{|c_J-y_2|}\Big)^\beta \frac{\Norm{b^1_L(\cdot,y_2)}{L^2(\mu_n)}^2}{\lambda_m(c_J,|c_J-y_2|)}\ud\mu_m(y_2) \\
&\lesssim \int \Big(\frac{\ell(L)}{d(\hat L, y_2)}\Big)^\beta \frac{\Norm{b^1_L(\cdot,y_2)}{L^2(\mu_n)}^2}{\lambda_m(x_{\hat L},d(\hat L, y_2))}\ud\mu_m(y_2) \\
&\lesssim \int_{\R^n} \int_{\supp b^1_L(y_1,\cdot)} \Big(\frac{\ell(L)}{d(\hat L,y_2)}\Big)^\beta \frac{d\mu_m(y_2)}{\lambda_m(x_{\hat L},d(\hat L,y_2))}\,d\mu_n(y_1).
\end{align*}
Since these bounds do not depend on $J$, the proof is completed by observing that
\begin{equation*}
  \sum_{J\in\mathcal{C}_L}\mu_m(J)=\mu_m(\hat L).\qedhere
\end{equation*}
\end{proof}

It remains to sum the bound of the Lemma over all $L\in\mathcal{D}_m'$. (Note that $F_L=\varnothing=\supp b^1_L$ if there is no dyadic interval $F$ such that $F\times L\subset\tilde\Omega$.) Thus we need to estimate
\begin{equation}\label{eq:b1Final}
\begin{split}
  \sum_{\substack{K\times L\in\mathcal{D}'\\
  K\times L\subset\Omega}} &\sum_{\substack{I\in\mathcal{C}_K\\ J\in\mathcal{C}_L}}
    \abs{\pair{h_{I\times J}}{Tb^1_L}}^2\\
  &  \lesssim\sum_{L\in\mathcal{D}_m'} \mu_m(\hat{L})\int_{\R^n} \int_{\supp b^1_L(y_1,\cdot)}
   \Big(\frac{\ell(L)}{d(\hat L,y_2)}\Big)^\beta
   \frac{d\mu_m(y_2)}{\lambda_m(x_{\hat L},d(\hat L,y_2))}\,d\mu_n(y_1) \\
  &  =\iint \sum_{L:\hat{L}\ni x_2}\int_{\supp b^1_L(y_1,\cdot)}\Big(\frac{\ell(L)}{d(\hat L,y_2)}\Big)^\beta
  \frac{\ud\mu_m(y_2)}{\lambda_m(x_{\hat L},d(\hat L,y_2))}\ud\mu(y_1,x_2),
\end{split}
\end{equation}
where the last equality follows by writing $\mu_m(\hat L)=\int \chi_{\hat L}(x_2)\ud\mu_m(x_2)$ and reorganizing.

Note that the outermost double integral may be restricted to $\tilde\Omega$. Indeed, let $x_2\in L$ and $\supp b^1_L(y_1,\cdot)$ be nontrivial. Then necessarily $y_1\in F_L$, and thus $(y_1,x_2)\in F_L\times L\subset\tilde\Omega$.

For each fixed $(y_1,x_2)\in\tilde\Omega$, let $T=T(y_1,x_2)$ denote the maximal $L\in\mathcal{D}_m'$ appearing in \eqref{eq:b1Final} with the property that $x_2\in \hat L$ and $\supp b^1_L(y_1,\cdot)\neq\varnothing$, so in particular $y_1\in F_L$. The existence of these maximal cubes $L$ can be seen as follows: Since $\Omega$ is bounded, we can obviously add the implicit restriction $\ell(L)\leq\textup{diam}(\Omega)$ to the summation on the first line, and then all other lines, in \eqref{eq:b1Final}. Thus the size of the relevant cubes $L$ is bounded from above, and then the existence of maximal cubes is a standard property of dyadic grids.

By definition, we have $F_T\times T\subset\tilde\Omega$ and $y_1\in F_T$. So in particular $\{y_1\}\times T\subset\tilde\Omega$. Since $\supp b^1_L\subset\tilde\Omega^c$, it follows that $\supp b^1_L(y_1,\cdot)\subset T^c$, so any $y_2$ in the integral satisfies $y_2\in\supp b^1_L(y_1,\cdot)\subset T^c$.

If  $z \in \hat L$, then there is a good cube $J \in \mathcal{D}_{m}$ so that $z \in J \subset L \subset T$ and $\ell(J) = 2^{-r}\ell(L) \le 2^{-r}\ell(T)$.
The goodness implies that
\begin{displaymath}
|z- y_2| \ge d(J, T^c) \ge \ell(J)^{\gamma_m}\ell(T)^{1-\gamma_m} = 2^{-r \gamma_m} \ell(L)^{\gamma_m}\ell(T)^{1-\gamma_m}.
\end{displaymath}
We may conclude that every $y_2$ in the integral satisfies the bound $d(\hat L, y_2) \ge 2^{-r \gamma_m} \ell(L)^{\gamma_m}\ell(T)^{1-\gamma_m}$.
So we may estimate
\begin{align*}
&\int_{\supp b^1_L(y_1,\cdot)}\frac{d(\hat L, y_2)^{-\beta}d\mu_m(y_2)}{\lambda_m(x_{\hat L},d(\hat L,y_2))}\\
&\le \sum_{j=0}^{\infty} \int_{\{y_2:\,  2^j2^{-r \gamma_m} \ell(L)^{\gamma_m}\ell(T)^{1-\gamma_m}\le d(\hat L, y_2) < 2^{j+1} 2^{-r \gamma_m} \ell(L)^{\gamma_m}\ell(T)^{1-\gamma_m}\}} 
\frac{d(\hat L, y_2)^{-\beta}d\mu_m(y_2)}{\lambda_m(x_{\hat L},d(\hat L,y_2))} \\
&\lesssim \sum_{j=0}^{\infty} (2^j2^{-r \gamma_m} \ell(L)^{\gamma_m}\ell(T)^{1-\gamma_m})^{-\beta} 
\frac{\mu_m(B(x_{\hat L}, [2+2^{r\gamma_m}] \cdot 2^j2^{-r \gamma_m} \ell(L)^{\gamma_m}\ell(T)^{1-\gamma_m}))}{\lambda_m(x_{\hat L},2^j2^{-r \gamma_m} \ell(L)^{\gamma_m}\ell(T)^{1-\gamma_m})} \\
&\lesssim \ell(L)^{-\beta \gamma_m} \ell(T)^{-\beta(1-\gamma_m)}.
\end{align*}
This gives us the bound
\begin{align*}
\sum_{L:\hat{L}\ni x_2}\int_{\supp b^1_L(y_1,\cdot)}\Big(\frac{\ell(L)}{d(\hat L,y_2)}\Big)^\beta \frac{\ud\mu_m(y_2)}{\lambda_m(x_{\hat L},d(\hat L,y_2))} 
&\lesssim \sum_{L:\,x_2\in L\subset T} \Big(\frac{\ell(L)}{\ell(T)}\Big)^{\beta(1-\gamma_m)} \\
&= \sum_{j=0}^{\infty} 2^{-\beta(1-\gamma_m)j} \lesssim 1.
\end{align*}


Substituting back to \eqref{eq:b1Final} and recalling that the outer double integral is restricted to $\tilde\Omega$, we deduce that
\begin{equation*}
  \sum_{K\times L\subset\Omega}
 \sum_{\substack{I\in\mathcal{C}_K\\ J\in\mathcal{C}_L}} \abs{\pair{h_{I\times J}}{Tb^1_L}}^2
  \lesssim\iint_{\tilde\Omega}\ud\mu(y_1,x_2)
  =\mu(\tilde\Omega)\lesssim\mu(\Omega),
\end{equation*}
as required.

\subsection{Analysis of $b^2_L$}
Here we sum the multiple series in a different order:
\begin{lem}
For a fixed $K \in \mathcal{D}_n'$, the sum over the other variables is estimated as
\begin{equation*}
  \sum_{\substack{L\in\mathcal{D}_m'\\ K\times L\subset\Omega}}
  \sum_{\substack{I\in\mathcal{C}_K\\ J\in\mathcal{C}_L}}
    \abs{\pair{h_{I\times J}}{Tb^2_L}}^2
   \lesssim\sum_{G\in\mathcal{G}_K}\mu(K\times G)\Big(\frac{\ell(K)}{\ell(\tilde{K}_G)}\Big)^{\alpha(1-\gamma_n)},
\end{equation*}
where $\mathcal{G}_K$ consists of the maximal $G\in\mathcal{D}_m'$ such that $K\times G\subset\Omega$, and $\tilde{K}_G\in\mathcal{D}_n'$ is the maximal cube such that $\tilde{K}_G\supset K$ and $\tilde{K}_G\times G\subset\tilde\Omega$.
\end{lem}

Before going to the proof, let us comment on the existence of the maximal cubes in the statement of the Lemma. Since $\Omega$ is bounded, there is an upper bound for $\ell(G)$ such that $K\times G\subset\Omega$, so $\mathcal{G}_K$ is immediately well defined. As for $\tilde{K}_G$, there again arises the possibility of an infinite increasing sequence of $K_k\supset K$ with $K_k\times G\subset\tilde\Omega$. As before, such a sequence will exhaust $\R^n$, thus $\R^n\times G\subset\tilde\Omega$, and accordingly we interpret $\tilde{K}_G:=\R^n$, $\ell(\tilde{K}_G):=\infty$ in this case. Note that the corresponding terms then vanish on the right of the asserted estimate.

\begin{proof}
Clearly every $L$ as in the left is contained in a unique $G\in\mathcal{G}_K$; thus
\begin{equation*}
  \sum_{\substack{L\in\mathcal{D}_m'\\ K\times L\subset\Omega}}
  \sum_{\substack{I\in\mathcal{C}_K\\ J\in\mathcal{C}_L}}
    \abs{\pair{h_{I\times J}}{Tb^2_L}}^2
  =   \sum_{I\in\mathcal{C}_K} \sum_{G\in\mathcal{G}_K} \sum_{\substack{L\in\mathcal{D}_m'\\ L\subset G}}\sum_{J\in\mathcal{C}_L}
    \abs{\pair{h_{I\times J}}{Tb^2_L}}^2.
\end{equation*}
In this sum, we observe that $\tilde{K}_G\times L\subset\tilde{K}_G\times G\subset\tilde\Omega$, and hence by definition $\tilde{K}_G\subset F_L$.
Now, for every $G\in\mathcal{G}_K$, there are two possibilities:
\begin{itemize}
  \item There is at least one $L\subset G$ such that $F_L\neq\R^n$. Then $\tilde{K}_G\subset F_L$ is a proper dyadic cube.
  \item For all $L\subset G$, we have $F_L=\R^n$. But then for all these $L$, we are in the degenerate case with $b^2_L\equiv 0$. Hence the part of the sum corresponding to these cubes $G$ will vanish, and can be ignored.
\end{itemize}
Thus, we can restrict the summation to the cubes $G\in\mathcal{G}_K$ for which $\tilde{K}_G\in\mathcal{D}_n'$ is a proper dyadic cube.

Now, for summation variables as above, we have that $I\subset\hat{K} \subset K \subset \tilde{K}_G$
and $F_L^c\supset\tilde{K}_G^c$ are disjoint, and we may use the partial kernel representation
\begin{equation*}
\begin{split}
  \pair{h_{I\times J}}{Tb^2_L}
  &=\iint h_I(x_1)\pair{h_J}{T_1(x_1,y_1)b(y_1,\cdot)}_2 \chi_{F_L^c}(y_1)\ud\mu_n(x_1)\ud\mu_n(y_1) \\
  &=\iint h_I(x_1)\pair{h_J}{B_I(x_1,y_1)}_2 \chi_{F_L^c}(y_1)\ud\mu_n(x_1)\ud\mu_n(y_1),
\end{split}
\end{equation*}
where
\begin{equation*}
  B_I(x_1,y_1)
  :=[T_1(x_1,y_1)-T_1(c_I,y_1)]b(y_1,\cdot).
\end{equation*}
We have
\begin{displaymath}
|c_I - y_1| \ge d(I, F_L^c) \ge d(I, K^c) \ge \ell(I)^{\gamma_n}\ell(K)^{1-\gamma_n} \ge \ell(I) \ge 2|x_1-c_I|,
\end{displaymath}
and therefore we may estimate
\begin{equation*}
  \Norm{T_1(x_1,y_1)-T_1(c_I,y_1)}{CZ(\mu_m)}\lesssim \Big(\frac{\ell(I)}{|c_I-y_1|}\Big)^\alpha\frac{1}{\lambda_n(c_I,|c_I-y_1|)}.
\end{equation*}
Since $b(y_1,\cdot)\in L^\infty(\mu_m)$ is mapped into $\BMO_3^2(\mu_m)$ by the Calder\'on--Zygmund operator above, we also have
\begin{align*}
  \Norm{B_I(x_1,y_1)}{\BMO_3^2(\mu_m)} &\lesssim \Norm{T_1(x_1,y_1)-T_1(c_I,y_1)}{CZ(\mu_m)} \\
  &\lesssim \Big(\frac{\ell(I)}{|c_I-y_1|}\Big)^\alpha\frac{1}{\lambda_n(c_I,|c_I-y_1|)}.
\end{align*}

Estimating $\chi_{F_L^c}(y_1)\leq \chi_{\tilde K_G^c}(y_1)$, we then deduce that with a fixed $I \in \mathcal{C}_K$ there holds
\begin{align*}
  \sum_{G\in\mathcal{G}_K}& \sum_{\substack{L\subset G\\ J\in\mathcal{C}_L}}
    \abs{\pair{h_{I\times J}}{Tb^2_L}}^2 \\
   &\leq\Big(\iint \abs{h_I(x_1)}\Big[ \sum_{G\in\mathcal{G}_K} \chi_{\tilde K_G^c}(y_1) \sum_{\substack{L\subset G\\ J\in\mathcal{C}_L}}
       \abs{\pair{h_J}{B_I(x_1,y_1)}_2}^2\Big]^{1/2} \ud\mu_n(x_1)\ud\mu_n(y_1)\Big)^2.
\end{align*}
We then note that
\begin{align*}
 \sum_{\substack{L\subset G\\ J\in\mathcal{C}_L}} \abs{\pair{h_J}{B_I(x_1,y_1)}_2}^2
 &\le \mathop{\mathop{\sum_{J \textup{ good}}}_{J \subset G}}_{\ell(J) \le 2^{-r}\ell(G)} \abs{\pair{h_J}{B_I(x_1,y_1)}_2}^2 \\ 
 &\lesssim \mu_m(G) \|B_I(x_1,y_1)\|_{\BMO_3^2(\mu_m)}^2
\end{align*}
by Lemma \ref{lem:BMO1}. This gives us that
\begin{align*}
&\sum_{G\in\mathcal{G}_K} \sum_{\substack{L\subset G\\ J\in\mathcal{C}_L}} \abs{\pair{h_{I\times J}}{Tb^2_L}}^2 \\
&\lesssim \Big(\iint \abs{h_I(x_1)} \|B_I(x_1,y_1)\|_{\BMO_3^2(\mu_m)} \Big[ \sum_{G\in\mathcal{G}_K} \mu_m(G) \chi_{\tilde K_G^c}(y_1)
\Big]^{1/2} \ud\mu_n(x_1)\ud\mu_n(y_1)\Big)^2 \\
&\lesssim \mu_n(I) \Big( \int \Big(\frac{\ell(I)}{|c_I-y_1|}\Big)^\alpha\frac{1}{\lambda_n(c_I,|c_I-y_1|)}
\Big[ \sum_{G\in\mathcal{G}_K} \mu_m(G) \chi_{\tilde K_G^c}(y_1) \Big]^{1/2} \ud\mu_n(y_1)\Big)^2.
\end{align*}
The integral squared is estimated by
\begin{align*}
\Big( \int_{I^c} &\Big(\frac{\ell(I)}{|c_I-y_1|}\Big)^\alpha\frac{\ud\mu_n(y_1)}{\lambda_n(c_I,|c_I-y_1|)} \Big)  \\
&\times \Big( \int \Big(\frac{\ell(I)}{|c_I-y_1|}\Big)^\alpha\frac{1}{\lambda_n(c_I,|c_I-y_1|)}  \sum_{G\in\mathcal{G}_K} \mu_m(G) \chi_{\tilde K_G^c}(y_1)  \ud\mu_n(y_1) \Big),
\end{align*}
which is further dominated by
\begin{align*}
\sum_{G\in\mathcal{G}_K} &\mu_m(G) \int_{\tilde K_G^c} \Big(\frac{\ell(I)}{|c_I-y_1|}\Big)^\alpha\frac{\ud\mu_n(y_1)}{\lambda_n(c_I,|c_I-y_1|)} \\
&\lesssim \sum_{G\in\mathcal{G}_K} \mu_m(G) \int_{y_1:\, d(\hat K, y_1) \ge 2^{-r\gamma_n}\ell(K)^{\gamma_n}\ell(\tilde K_G)^{1-\gamma_n}} 
\Big(\frac{\ell(K)}{d(\hat K, y_1)}\Big)^\alpha\frac{\ud\mu_n(y_1)}{\lambda_n(x_{\hat K},d(\hat K, y_1))}  \\
&\lesssim \sum_{G\in\mathcal{G}_K} \mu_m(G) \Big(\frac{\ell(K)}{\ell(\tilde K_G)}\Big)^{\alpha(1-\gamma_n)}.
\end{align*}
These estimates follow completely analogously to the ones in the previous subsection and utilize goodness in an essential way.

We have established the bound
\begin{align*}
 \sum_{I\in\mathcal{C}_K} \sum_{G\in\mathcal{G}_K} \sum_{\substack{L\subset G\\ J\in\mathcal{C}_L}} \abs{\pair{h_{I\times J}}{Tb^2_L}}^2
&\lesssim  \sum_{I\in\mathcal{C}_K} \mu_n(I) \sum_{G\in\mathcal{G}_K} \mu_m(G) \Big(\frac{\ell(K)}{\ell(\tilde K_G)}\Big)^{\alpha(1-\gamma_n)} \\
&= \mu_n(\hat K) \sum_{G\in\mathcal{G}_K} \mu_m(G) \Big(\frac{\ell(K)}{\ell(\tilde K_G)}\Big)^{\alpha(1-\gamma_n)} \\
&\le \sum_{G\in\mathcal{G}_K} \mu(K \times G) \Big(\frac{\ell(K)}{\ell(\tilde K_G)}\Big)^{\alpha(1-\gamma_n)},
\end{align*}
which ends the proof of the lemma.
\end{proof}

It remains to estimate the right side of the Lemma summed over $K \in \mathcal{D}_n'$:
\begin{align*}
\sum_{K \in \mathcal{D}_n'} \sum_{G\in\mathcal{G}_K} \mu(K \times G) \Big(\frac{\ell(K)}{\ell(\tilde K_G)}\Big)^{\alpha(1-\gamma_n)} 
&\le \mathop{\sum_{K \times G \subset \Omega}}_{\textup{2-maximal}} 2^{-\alpha(1-\gamma_n)\operatorname{emb}_1(K \times G;\Omega)}\mu(K \times G) \\
\le 2 \sum_{k=0}^{\infty} 2^{-\alpha(1-\gamma_n)k} \times \mu(\Omega) \lesssim \mu(\Omega).
\end{align*}
Here we used that by the definition of $\mathcal{G}_K$, the rectangles $K\times G\subset\Omega$ are 2-maximal. 
Second, by the definition of $\tilde K_G$ we have
\begin{equation*}
  \frac{\ell(K)}{\ell(\tilde{K}_G)}= 2^{-\operatorname{emb}_1(K\times G;\Omega)}.
\end{equation*}
Lastly, we utilized Theorem~\ref{thm:Journe}.  This completes the proof of Theorem~\ref{thm:BMO}.

\end{document}